\documentclass[12pt]{amsart}
\usepackage{amsmath}

\usepackage{a4wide,enumerate,xcolor,graphicx}
\usepackage{amsmath,comment,algorithm}
\allowdisplaybreaks
\usepackage{tabularx}
\usepackage{multirow}
\usepackage{graphicx}
\usepackage{subcaption}
\usepackage{geometry}
\usepackage{algorithm}
\usepackage{algpseudocode}
\usepackage{amsthm}
\usepackage{amssymb}
\usepackage{bbm}
\usepackage{setspace}
\usepackage{cite}
\usepackage{color}
\usepackage{physics}
\usepackage[section]{placeins}
\usepackage{tikz}
\usetikzlibrary{shapes.arrows,shadings,shadows,decorations.pathreplacing,decorations.markings,arrows.meta,calligraphy}

\usepackage[dvipsnames]{xcolor}
\usetikzlibrary{positioning}

\usepackage{hyperref}           % hyperref should always come last

\usepackage{graphicx} % Required for inserting images
\usepackage{amsmath,amssymb,a4wide,amsthm,mathtools}
\usepackage{hyperref}
\usepackage{enumitem}
\usepackage{xcolor}
\usepackage{comment}
\usepackage{graphicx}
\usepackage{csquotes}
\usepackage{chngcntr}
\usepackage{algorithm}

\newtheorem{theorem}{Theorem}
\newtheorem{assumptions}{Assumptions}
\newtheorem{proposition}{Proposition}
\newtheorem{remark}{Remark}
\def\b{\mathbf{b}}
\def\dom{\Omega}
\def\bou{{\partial\dom}}
\def\llh{{\Phi_\mathrm{AAO}}}
\def\sigep{{\Sigma_\varepsilon}}
\def\sigres{{\Sigma_\text{res}}}

\newcommand{\R}{{\mathbb R}}
\newcommand{\N}{{\mathcal N}}
\newcommand{\embed}{\hookrightarrow}

\def\dom{\Omega}

\def\Lc{\mathcal{L}}
\def\Rc{\mathcal{R}}

\begin{document}

\title[Bayesian Inverse Problems in All-at-Once Framework]{PDE-based Bayesian Inverse Problems in All-at-Once Framework}% with application to semiconductor devices}

\author[]{}
\address{}
\email{}

\author[]{Tram Thi Ngoc Nguyen$^{1}$, Fabian Schneider$^{2, 3}$, and Leila Taghizadeh$^{2,\ast}$}
\address{$^{1}$Department of Applied Mathematics and Theoretical Physics,
University of Cambridge, Cambridge, England}
\address{$^{2}$Institute of Analysis and Scientific Computing, TU Wien, Vienna, Austria}
\address{$^{3}$School of Engineering Sciences, LUT University, Lappeenranta, Finland}

\thanks{$^\ast$Corresponding author: leila.taghizadeh@tuwien.ac.at}

\date{\today}

 % \thanks{\textcolor{blue}{Leila's FWF grant + grant supporting Fabian + VSM support for Tram's visit + FWF support for Leila's visit}} 

%%%%%%%%%%%%%%%%%%%%%%%%%%%%%%%%%%%%%%%%%%%%%%%%%%%%%%%%%%%%%%%5
\begin{abstract} 
We propose a modern all-at-once framework for PDE-based Bayesian inversion by introducing the PDE residual as probabilistic information. In this framework, the PDE residual is treated as a soft constraint with a residual covariance and  is directly incorporated into the likelihood function. 
The new hybrid likelihood function %thus introduces a framework where the redundant 
eliminates PDE solves, thus reducing expensive computational costs, while performing rigorous Bayesian inference.
Theoretically, we establish the well-posedness of the posterior with uncertainty quantification, and numerically, we propose two efficiency inversion algorithms based on two posterior sampling methods.
In particular, one algorithm is based on Metropolis-within-Gibbs MCMC method %for posterior sampling 
with a PDE-informed block-diagonal Gaussian prior, and the other is based on a joint sampling strategy.
Numerical experiments for a semiconductor inverse problem confirm that the proposed Bayesian all-at-once approach yields a speedup of approximately a factor of two, %, and more with tightened prior on the PDE-solution, 
while having a comparable reconstruction quality as the conventional reduced-based setting.

\end{abstract}
%%%%%%%%%%%%%%%%%%%%%%%%%%%%%%%%%%%%%%%%%%%%%%%%%%%%%%%%%%%%%%%

% \paragraph{AMS classification:}  
\subjclass[2000]{}
%%%%%%%%%%%%%%%%%%%%%%%%%%%%%%%%%%%%%%%%%%%%%%%%%%%%%%%%%%%%%%%
\maketitle
%\tableofcontents

%%%%%%%%%%%%%%%%%%%%%%%%%%%%%%%%%%%%%%%%%%%%%%%%%%%%%%%%%%%%%%%%%
\section{Introduction}\label{s:Intro}

The Bayesian approach to inverse problems treats the unknown parameters as random variables, representing uncertainty explicitly through probability distributions  \cite{stuart2010inverse,kaipio2005statistical,SmithRalphUQ}. Prior information is incorporated through a prior probability distribution, providing a principled form of regularization, while observational data are introduced through the data likelihood function. Together, the prior and the likelihood are combined via Bayes' theorem to yield the posterior distribution, which encodes the updated state of knowledge about the unknown parameters given the observed data \cite{SmithRalphUQ}. Unlike deterministic approaches, the Bayesian framework quantifies uncertainty of the inversion by estimating the posterior probability distribution of the unknown and the
standard deviation of the posterior. This stochastic approach for inverse problems has been developed fundamentally for various forward model problems, such as PDEs \cite{buithanh2013computational,dashti2011uncertainty,cui2024multilevel,taghizadeh2020bayesian,stadlbauer2019bayesian,scheichl2021numerical}, and for many applications ranging from inverse modeling of COVID-19 pandemic \cite{taghizadeh2020uncertainty} to medical imaging \cite{taghizadeh2020bayesian}.

Despite this essential progress, the practical application of Bayesian inversion for PDE-constrained problems remains limited by a significant computational bottleneck. 
Indeed,  Bayesian inversion is primarily carried out in a \emph{reduced formulation}. This means each evaluation of the likelihood requires solving the underlying (often nonlinear) PDE, and sampling-based methods such as Markov chain Monte Carlo (MCMC) typically require many thousands of such evaluations to adequately explore the posterior distribution \cite{stuart2010inverse,bui2013computational}. 
This cost is further compounded when the unknown
parameter is a high- or infinite-dimensional spatial field, %making conventional Bayesian inversion computationally prohibitive for many practical problems,
as sampling efficiency tends to degrade with dimension \cite{cui2024multilevel,bui2013computational}. This computational burden motivates the development of methods that reduce the cost of forward PDE solves required for Bayesian inference.

In this view, an \emph{all-at-once formulation}  mitigates this exact issue.
All-at-once formulations originated in optimization with PDE constraints \cite{Taasan91,KunischSachs,KupferSachs,LeibfritzSachs,LeHe16,orozco-ghattas-97b,shenoy-heinkenschloss-cliff-98}. This concept was subsequently studied in the context of deterministic inverse problems \cite{BurgerMuehlhuberIP,BurgerMuehlhuberSINUM,HaAs01,KKV14b,aao16,BonazzoliHaddarVu}, and more recently specifically for time-dependent inverse problems \cite{Kaltenbacher:17,nguyen19,nguyen24,nguyen25}. The all-at-once approach offers an alternative that bypasses solving the PDE exactly (or up to very high precision). Instead, it inverts the PDE state and the unknown parameter simultaneously by minimizing the data misfit together with the PDE residual. The idea of minimizing the PDE residual has recently also gained popularity in the context of PDE learning \cite{Raissi18,Raissi,holler22learning_parameter_id, KaNg2022}. A theoretical comparison in terms of convergence guarantees between all-at-once and classical reduced inversion  can be found in \cite{TCC21}, and real-life applications are discussed, e.g.~in magnetic particle imaging \cite{KNSW,NguyenWald:2022}.

This work combines the all-at-once paradigm with Bayesian inverse problems, which can be viewed as a rather modern approach. Indeed, while the transfer of the all-at-once formulation from optimization to deterministic inverse problems  has been intensively studied, the  extension from deterministic to Bayesian inversion remains much less explored. To the best of the authors' knowledge, \cite{schlintl2021all} is the only existing work in this direction and serves as an inspiration for our research.

Our key contributions  are summarized as follows:
\begin{itemize}
\item 
Inspired by \cite{schlintl2021all}, we extend the all-at-once formulation from deterministic to Bayesian inversion for general elliptic PDE framework. Our framework  provides a full probabilistic treatment, including numerical realization and  uncertainty quantification. 
\item Within this all-at-once framework, the Bayesian posterior is defined via a hybrid likelihood that combines data and PDE residual. This bypasses the need for repeated high-fidelity PDE solves, which is  expensive in PDE-based Bayesian inversion.

\item To place this in a rigorous Bayesian framework, we establish the well-posedness of the AAO Bayesian posterior as well as its stability in the Hellinger distance with respect to perturbations of the data. We moreover propose two posterior sampling methods -- a Metropolis-within-Gibbs MCMC approach and a joint MCMC sampler -- with a PDE-informed joint prior. 

\item Finally,  
we present numerical experiments for parameter identification first for the benchmark a-b-c-s PDE problem, and then in a  doping profile inverse problem for semiconductor devices. Through a comparison with the conventional reduced formulation,  the all-at-once Bayesian method  shows a significant reduction in computational cost, while maintaining high reconstruction quality.
\end{itemize}

\subsection{The problem setting}\label{ss:problem setting}
We study the nonlinear stationary PDE with  unknown  spatially dependent parameter $\theta$ in the general form
\begin{equation}\label{original-pde}
\begin{split}
&F(\theta,u)=0 \qquad F: X\times U \to W^*
\end{split}
\end{equation}
Here, $\theta$ typically appears as a source function, diffusion parameter, potential term, etc., while $u$ denotes the state; both are specified in the physical domain $\Omega\subset \mathbb{R}^d, d\leq 3$. The nonlinear model $F$ represents spatial derivatives and other nonlinear functions in $u$ as well as boundary condition; for example, $F=(-\Delta u +\theta u+u^3; u|_{\partial\Omega})$.
Above $U$, $X$ are, respectively, the state and parameter spaces, while $W^*$, as the dual of $W$, denotes the residual space of the PDE. In this article, we consider the Sobolev space framework, which is a convention choice for elliptic PDEs.

In order to identify the parameter $\theta$ in \eqref{original-pde},  additional information on the state is acquired from finite measurements
\begin{align}\label{L}
L:U\to \R^m(=:Y) \qquad Lu=y,
\end{align}
where $L$ is a linear bounded operator.
In this sense, well-posedness of $L$ requires  each of its component lives in the dual space of $U$, i.e. $L\in (U^*)^m$ (c.f.~\cite{Aarset_2025}). 
 Moreover, we assume the measurement noise model is
\begin{equation}\label{eq:observation model}
    y = Lu + \varepsilon,
\end{equation}
where
$\varepsilon=(\varepsilon_1,\ldots, \varepsilon_m)^T\sim \mathcal{N}(0,\Sigma_\varepsilon)$ is an additive Gaussian noise vector with mean zero and covariance matrix $\Sigma_\varepsilon\in \R^{m\times m}$.

\subsubsection{Classical Reduced Setting }\label{ss:deterministic Reduced setting}
Conventionally, one constructs the reduced forward map
\begin{align}\label{G}
 G=L\circ S: X\to Y \qquad G(\theta)=y
\end{align}
by composing the parameter-to-state (or PDE solution) map
\begin{align}\label{S}
S:X\to  U \qquad S(\theta)= u \text{ solves PDE }\eqref{original-pde}
\end{align}
with the observation operator $L$ defined in \eqref{L}. 
Then the parameter inversion task  can be formulated as the minimization problem
\begin{align}
    \min_{\theta\in X} \|G(\theta)-y\|^2_Y + \Rc(\theta)
\end{align}
with suitable regularizer $\Rc:X\to \R^+$, encoding some prior information about the unknown parameter. For an overview on general regularization techniques, we refer to \cite{Kirsch,EHNBuch,KalNeuSch08}.

\subsubsection{All-At-Once setting}

In contrast to the composed map \eqref{G}, the all-at-once suggests to construct the high-dimensional forward map 
\begin{align}\label{G-aao}
\begin{split}
%\mathbb{G}:X\times\Vc\to \Uc^*\times H\times\Yc \qquad 
\mathbb{G}(\theta,u):=
\begin{pmatrix}
Lu\\
F(\theta,u)
\end{pmatrix}
=\begin{pmatrix}
y\\0
\end{pmatrix}
\end{split} \in  Y \times W^*
\end{align}
coupling the PDE residual in parallel with the data misfit. 
 Importantly, the state $u$ is now considered as an independent unknown together with the parameter $\theta$.  This means that $u$ does not necessarily live in the solution manifold of the PDE \eqref{original-pde} but rather on the whole vector space $U$.  
 Another aspect is that compared to the reduced formulation, the PDE residual space $W^*$ now plays an explicit role, leading to well-posedness and further treatment in the new setting.

The parameter identification in this setting then minimizes the all-at-one misfit of the form
\begin{align}
\min_{(\theta,u)\in X\times U}\|\mathbb{G}(\theta,u)-(y, 0)\|^2_{ Y \times W^*} + \Rc(\theta,u).
\end{align}
We emphasize that since this formulation does not require the parameter-to-state map $S$ as in \eqref{S}, it bypasses the need for coupling the inversion implementation with a nonlinear PDE solver. This particular feature makes the all-at-once inversion very efficient with low computational cost, while maintaining the reconstruction quality, as shall be shown in Section \ref{s:numerics}.

\section{Bayesian Inversion in  All-at-Once Framework} 
In the all-at-once regime, as the PDE solution $u$ is inferred simultaneously with the unknown parameter $\theta$, we need to develop a joint Bayesian estimation. 
More precisely, the inference is performed jointly on the state and parameter as an unknown joint random vector, namely $(\theta, u)\in X\times U$. In this section, we first present a Bayesian formulation in the conventional reduced setting (Section \ref{sec:bayes-red}), paving the way to the new all-at-once framework with full description for prior, hybrid likelihood function, and posterior setting (Section \ref{ss:AAO-Bayesian}).

\subsection{Bayesian Inversion Formulation in the Reduced Setting}\label{sec:bayes-red} 
In the reduced setting, as discussed in \eqref{G}, the forward operator  $G=L\circ S: X\to Y$ composes the parameter-to-state map $S$ with the finite observation map $L$. 
Bayesian inversion assumes the unknown parameter $\theta\in X$ is an infinite random variable/field and observation $y\in Y= \R^m$ is perturbed by finite Gaussian noise.  The statistical model for the observational error is thus defined by
\begin{equation}\label{eq:observation model2}
    y = G(\theta) + \varepsilon \qquad \varepsilon=(\varepsilon_1,\ldots, \varepsilon_m)^T\sim \mathcal{N}(0,\Sigma_\varepsilon), \quad \Sigma_\varepsilon:=\epsilon^2 I\in \R^{m\times m},
\end{equation}
where
$\varepsilon>0$ is an additive white Gaussian noise vector with zero mean and covariance matrix $\Sigma_\varepsilon$.

The goal of Bayesian inverse problems  is to find a posterior probability distribution $p(\theta|y)$ of the parameter of interest, given noisy measurement data $y$.  Furthermore,  
we assume a prior probability distribution $p(\theta)$ on the unknown parameter, according to some prior knowledge. Then the  posterior probability distribution of the unknown parameter is given by Bayes' rule as
\begin{equation}
    p(\theta|y) \propto p(y|\theta)p(\theta).
\end{equation}
The Bayesian log-likelihood function in the reduced setting $\Phi_\text{RED}\colon X\times Y\to \R$ is defined by the data misfit
\begin{equation}\label{eq:log-likelihood-reduced}
    \Phi_\text{RED}(\theta; y) :=\frac12\|y-G(\theta)\|^2_{\Sigma_\varepsilon},
\end{equation}
modeling the observed data noise using a Gaussian distribution. This distribution has zero mean and covariance matrix $\Sigma_\varepsilon$ acting on the finite-dimensional data space $Y$. Moreover, assuming a Gaussian prior distribution for the parameter field $\theta$ with mean $\theta_0$ and covariance operator $\Sigma_{\theta_0}$, the full Bayesian log-posterior function in the reduced setting becomes
\begin{equation}
  J_\text{RED}(\theta):= 
  \frac12\|y-G(\theta)\|^2_{\Sigma_\varepsilon} 
  +\frac12\|\theta - \theta_0\|^2_{\Sigma_{\theta_0}}.
\end{equation}
Here, one looks for the unknown $\theta\in X$ in some Cameron-Martin space $X_\sigma:=\text{Range} (\Sigma_{\theta_0}^{1/2}) \subset X$  equipped with  inner
product $(\cdot,\cdot)_{X_\sigma}:= (\Sigma_{\theta_0}^{-1/2}\cdot,\Sigma_{\theta_0}^{-1/2}\cdot)$  associated with the Gaussian prior \cite{stuart2010inverse}.
\begin{remark}
    In the classical reduced inversion, the PDE residual is enforced to be exactly zero as a hard constraint. Hence, each evaluation of $S(\theta)=u$ requires solving the PDE \eqref{original-pde} with high precision.
\end{remark}
Completing the reduced setting, we now introduce the new all-at-once framework for Bayesian inversion. Our plan is to detail the associated the prior, likelihood and  posterior.

\subsection{Bayesian Inversion Formulation in the All-at-once setting}\label{ss:AAO-Bayesian}
In contrast to the reduced setting, the PDE residual $R:=F(\theta, u)$ in the all-at-once regime is treated a soft constraint, meaning that we do not require it to be zero, instead sufficiently small.
In particular, we consider this residual as a \emph{virtual observation/data} that only in the ideal observation scenario attains $R = 0$.
We then place a Gaussian distribution with mean zero and covariance on $R$ and formulate
\begin{equation}\label{eq:residual model}
    R= F(\theta, u) + \eta\in W^*, \qquad \eta\sim \mathcal{N}(0,\Sigma_{\text{res}}).
\end{equation}
Here, the additive virtual noise $\eta$ 
is Gaussian  with mean zero and  residual covariance operator $\Sigma_{\text{res}}$. Similar to the prior, this forms a Cameron-Martin space $W^*_\sigma \subset W^\ast$ associated with the residual field. 

The PDE residual $R$ 
%as a Gaussian noise 
has a Mat\'ern-type covariance operator 
\begin{align}\label{cov-res}
\Sigma_{\text{res}} = \epsilon_\text{res}^2(-\Delta + I)^{-s}:W^*\to W^*, \qquad s>d/2
\end{align}
with $\epsilon_\text{res}^2>0$ and well-defined inverse elliptic operator \cite{Evans}. By elliptic regularity, $\Sigma_{\text{res}}$
gains $2s$ orders of smoothness, so the associated Cameron--Martin space
$
W^*_\sigma := \mathrm{Range}\big(\Sigma_{\text{res}}^{1/2}\big) 
$
is $s$ orders smoother than the base residual space $W^*$. And $\sigres$ is a well-defined trace class operator under appropriate choice of $W^*$ and smoothness $s$. 
As an example, with $W^* = H^{-1}(\Omega)$ and $d=2$, taking $s>1$, i.e.\ $\Sigma_{\text{res}} = \epsilon_\text{res}^2(-\Delta + I)^{-s}$, or
$\Sigma_{\text{res}} = \epsilon_\text{res}^2(-\Delta)^{-s}$ under Dirichlet boundary conditions, yields
$W^*_\sigma = H^{-1+s}(\Omega)$, and $\sigres:W^*\to W^*$ is a trace class operator.

Then the Bayes' theorem yields the joint posterior probability distribution $p(\theta,u|y)$ over $\theta$ and $u$, given measured data $y$ and artificial data $R$, as
\begin{equation}\label{eq:joint posterior}
    p(\theta,u|y, R=0) \propto p(y, R=0\,|\theta, u)p(\theta,u),
\end{equation}
where $p(y, R=0\,|\theta, u)$ and $p(\theta,u)$ are respectively the hybrid likelihood and the joint prior distributions. Each component will be explained in the subsequent subsections.

\begin{remark}
    It is clear from the all-at-once formulation that $u$ is not necessarily a PDE solution associating to model parameter $\theta$. Therefore, one does not need to solve any PDE, significantly reducing computational cost. The  PDE residual contributes to the likelihood  and can be enforced to be small by choosing $\Sigma_{\text{res}}$ appropriately.
\end{remark}

\subsubsection{Likelihood function with a hybrid data space}\label{ss:AAO-likelihood}
%In  classical Bayesian inverse problems, the data space is often assumed to be a finite-dimensional space $Y\subset \R^m$, and the PDE is assumed to be a hard constraint by exactly enforcing $F(\theta, u)=0$, and thus it is not part of the likelihood function. But in the AAO-Bayesian setting, 
We now define a hybrid data space $Y \times W^*$, where again $Y=\R^m$ is the finite-dimensional data space,
and $W^*$ is the infinite-dimensional PDE residual space. And the all-at-once likelihood function will be defined on such a hybrid space, directly influencing the resulting posterior distribution. 

By construction, the Bayesian AAO setting composes two sources of information, namely the observed data and an implicit residual data. 
The two source of information are given respectively by 
the likelihood of the observed data
\begin{equation}\label{eq:likelihood dist1}
    p(y|\theta,u):= \mathcal{N}(L(u),\Sigma_\varepsilon),\quad  \text{with }\Sigma_\varepsilon \text{ as in } \eqref{eq:observation model2}%\Sigma_\varepsilon:=\epsilon^2I
\end{equation}
and the physics likelihood in the sense of PDE residual
\begin{equation}\label{eq:likelihood dist2}
    p(R=0\,|\theta,u):= \mathcal{N}(F(\theta, u),\Sigma_{\text{res}}) \quad \text{with }\sigres \text{ as in }\eqref{cov-res}.%\Sigma_{\text{res}}:=\epsilon_{\text{res}}^2I
\end{equation}
We stress on the fact that here, the state $u$ is treated as an independent unknown rather than %being obtained by solving the PDE for a given 
a state associating to $\theta$. 
This allows the data noise $\varepsilon$ and the residual noise $\eta$ to remain unrelated, and the two terms act on different, uncoupled evaluations. % of $(\theta,u)$, with the PDE constraint itself entering only through $\Phi_{\mathrm{res}}$.
We assume that these two noise sources  are conditionally independent given $(\theta,u)$. Under  this, the hybrid likelihood is defined by
\begin{equation}\label{eq:log-likelihhod func}
    p(y,R=0\,|\theta,u):=  p(y|\theta,u) p(R=0\,|\theta,u).
\end{equation}
This means that in the Bayesian AAO setting, in addition to the classical statistical measurement noise, the physical law also contains implicit uncertainty.
Note that the assumption of conditional independence of the two noise sources is natural here, and is  weaker  than unconditional independence of $y$ and $R$. Since both $L(u)$ and $F(\theta,u)$ are deterministic, once $(\theta,u)$ is fixed, conditioning on $(\theta,u)$ removes the dependence between $y$ and $R$ induced by the shared state $u$. This leaves only the requirement that the measurement noise $\varepsilon$ in \eqref{eq:observation model2}  and the virtual residual noise $\eta$ in \eqref{eq:residual model} are independent random variables. This is a natural modeling assumption, since $\varepsilon$ represents physical (e.g., sensor) measurement noise arising at the data acquisition stage, while $\eta$ is a synthetic device introduced to relax the PDE constraint into a soft, probabilistic form \eqref{eq:residual model}.

Translating \eqref{eq:likelihood dist1} and \eqref{eq:likelihood dist2} into  data-likelihood function
\begin{equation}\label{eq:likelihood function1}
  p(y|\theta,u) \propto\exp\bigg(-\frac12\|y-L(u)\|^2_{\Sigma_\varepsilon}\bigg) =  \exp\bigg(-\frac{1}{2\epsilon^2}\|y-L(u)\|_2^2\bigg)
\end{equation}
and residual-likelihood function
\begin{equation}\label{eq:likelihood function2}
  p(R=0\,|\theta,u) \propto\exp\bigg(-\frac12\|F(\theta,u)\|^2_{\Sigma_{\text{res}}}\bigg),% =  \exp\bigg(-\frac{1}{2\epsilon_{\text{res}}^2}\|F(\theta,u)\|_{2}^2\bigg).
\end{equation}
the hybrid log-likelihood function $\Phi_\mathrm{AAO} \colon X \times U\times Y \to \R$  can be formulated as
\begin{equation}\label{eq: log-likelihood for aao}
 \Phi_\mathrm{AAO}(\theta,  u; y) :=\Phi_\mathrm{data}(u; y) + \Phi_\mathrm{res}(\theta,  u)=\frac12\|y-L(u)\|^2_{\Sigma_\varepsilon}
+\frac12\|F(\theta,u)\|^2_{\Sigma_{\text{res}}}
\end{equation}
up to additive constants independent of the unknowns and data.
Hence, the log-likelihood function in the AAO setting is the sum of the data misfit $\Phi_\mathrm{data}$ and PDE residual $\Phi_\mathrm{res}$.

\subsubsection{Joint prior distribution} 
We now define the joint prior distribution over the unknown parameter and the state fields.  Recall that the state $u$ is treated as a random field, instead of a deterministic quantity as in the reduced setting. We thus propose the priors
\begin{equation}\label{eq:parameter prior}
 p(\theta, u) = p(\theta)p(u) \quad\text{with}\quad p(\theta)  := \N (\theta_0,\Sigma_{\theta_0}),\quad p(u) := \N(u_0,\Sigma_{u_0}).
\end{equation}
The parameter prior $p(\theta)$ is usually an informative Gaussian prior distribution, while  the mean $u_0$ in the state prior is chosen as the linearized state, meaning $u_0 := S'(\theta_0)$. 

The factorization $p(\theta, u) = p(\theta)p(u)$ assumed above is a modeling choice made at the level of the prior, before any data or PDE-residual information is incorporated.  
This assumption is carried through the rest of the construction, it underlies the marginals \eqref{eq:parameter prior} and the block-diagonal covariance
$\Sigma_0$  for discretization later (see Section~\ref{ss:prior sampling}).
This in turn justifies (i)
separate Karhunen--Lo\`eve expansions for $\theta$ and $u$, and (ii) the product Cameron--Martin space $X_\sigma \times U_\sigma$.

\subsubsection{Joint posterior distribution} 
We now have all the ingredients to formulate the AAO posterior distribution. Indeed, \eqref{eq:joint posterior} can be expressed as
\begin{equation*}\label{eq:Bayes AAO}
     p(\theta,u|y, R=0) \propto \exp(- \Phi_\mathrm{AAO}(\theta,  u; y))\mu_{\theta_0}(\theta)\mu_{u_0}(u).
\end{equation*}
which by inserting \eqref{eq: log-likelihood for aao}~\eqref{eq:parameter prior} becomes
\begin{equation}\label{eq: posterior equation}
  p(\theta,u|y, R=0) \propto \exp\bigg(
  -\frac12\|y-L(u)\|^2_{\Sigma_\varepsilon}
  -\frac12\|F(\theta,u)\|^2_{\Sigma_{\text{res}}} 
  -\frac12\|\theta - \theta_0\|^2_{\Sigma_{\theta_0}} 
  -\frac12\|u -u_0\|^2_{\Sigma_{u_0}} \bigg)
\end{equation} 
and the corresponding cost functional is
\begin{equation}
  J_\text{AAO}(\theta, u):= 
  \frac12\|y-L(u)\|^2_{\Sigma_\varepsilon}
  +\frac12\|F(\theta,u)\|^2_{\Sigma_{\text{res}}} 
  +\frac12\|\theta - \theta_0\|^2_{\Sigma_{\theta_0}} 
  +\frac12\|u -u_0\|^2_{\Sigma_{u_0}}.
\end{equation}

Since the likelihood now involves both measurements and the PDE residual, it is not immediately obvious that the resulting Bayesian posterior distribution is well-defined. Section~\ref{sec:Well-posedness} is thus dedicated to establish the well-posedness result.
Before proceeding, we briefly remark that the same question for the deterministic counterpart, in terms of existence and stability of the  minimizer under noise level can be found, e.g.~in \cite{KaNg2022}.

\section{Well-posedness of Bayesian Inversion in the All-at-Once Setting}\label{sec:Well-posedness}

First of all, we recall assumptions \cite[Assumption 2.6]{stuart2010inverse} on the general log-likelihood function $\Phi$ such that the associated Bayesian inverse problem is well-defined. We then  adapt these assumptions to the all-at-once setting with the residual-based log-likelihood  $\Phi_\mathrm{AAO}$ as defined in \eqref{eq: log-likelihood for aao}. This adaptation results in the following Assumption \ref{assump}.

\begin{assumptions}\label{assump}
%Let $X$, $Y$, and $U$ be Banach spaces. 
Suppose that the all-at-once log-likelihood function $\Phi_\mathrm{AAO} \colon  X \times U\times Y \to \R$ satisfy:
\begin{enumerate}
\item[\rm (i)] (Lower bound). For any $\alpha>0$ and $R>0$, there exists $M=M(\alpha,R)\in\R$ such that for all $(\theta,  u)\in X \times U$ and $\|y\|_{Y}<R$, one has
\begin{align*}
  \Phi_\mathrm{AAO}(\theta,  u; y) \ge M(\alpha,R)-\alpha\|(\theta,  u)\|^2_{X \times U}.
\end{align*} 
\item[\rm (ii)] (Upper bound). For any $R>0$, there exists $K(R)>0$ such that for all $(\theta,  u)\in X \times U$ and $y\in Y$  with $\max\{\|(\theta,  u)\|_{X \times U},\|y\|_{Y}\}<R$, it holds
\begin{align*}
  \Phi_\mathrm{AAO}(\theta,  u; y) \le K(R).
\end{align*}
\item[\rm (iii)] (Local Lipschitz continuity in $(\theta,  u)$).\\ For any $R>0$, there exists $L(R)>0$ such that for all $(\theta_1,  u_1),(\theta_2,  u_2)\in X \times U$ and $y\in Y$ with 
$$\max\{\|(\theta_1,  u_1)\|_{X \times U},\|(\theta_2,  u_2)\|_{X \times U},\|y\|_{Y}\}<R,$$
%\red{I think the second component of data is some $s\in W^*$, i.e.~$(y,s)$ instead of $(y,F(\theta,u))$. The same is for (iv).} Yes!\\
the following holds
\begin{align}\label{ass-lip-unknown}
  |\Phi_\mathrm{AAO}(\theta_1,  u_1;y)-\Phi_\mathrm{AAO}(\theta_2,  u_2;y)|\le L(R)\|(\theta_1,  u_1)-(\theta_2,  u_2)\|_{X \times U}.
\end{align}
\item[\rm (iv)] (Local Lipschitz continuity in $y$).\\ For any $\beta>0$ and $R>0$, there exists $C(\beta,R)>0$ such that for all $(\theta,  u)\in {X \times U}$ and $y_1,y_2\in Y$ with 
$$\max\{\|y_1\|_{Y},\|y_2\|_{Y}\}<R,$$
the following holds
\begin{align}\label{ass-lip-data}
  |\Phi_\mathrm{AAO}(\theta,  u;y_1)-\Phi_\mathrm{AAO}(\theta,  u;y_2)|
 \leq \exp\big(\beta\|(\theta,  u)\|^2+C(\beta,R)\big)\|y_1-y_2\|_{Y}.
\end{align}
\end{enumerate}
\end{assumptions}

We aim to establish well-posedness not only for a specific PDE, but for a general class of nonlinear equations as introduced in \eqref{original-pde}. To this end, we restrict ourselves to semilinear equations, a class that is still sufficiently broad to cover a wide range of applications \cite{Pao,Troeltzsch} and allows us to verify Assumption~\ref{assump}. 
In the following,  we denote by $X\embed Y$  the general continuous Sobolev embeddings,  by $C_{X\to Y}$ their operator norms. And the generic constant $C$ may take different values at each occurrence.

\begin{proposition}\label{prop:wellposed}
Consider the semilinear elliptic equation of the form
    \begin{align}\label{eq: abcs}
    F(\theta, u):= [B\theta]u +f(u)-s 
\end{align}
with the bilinear operator $B\in\Lc(X,\Lc(U,W^*))$, and the nonlinearity $f\in\mathcal{C}^1(\R)$ fulfills the growth condition
\begin{align}\label{f-growth}
|f(x)-f(y)|\leq C_f(1+|x|^\gamma +|y|^\gamma)|x-y| \quad\forall x,y\in\R, \text{with some } \gamma\geq0.
\end{align}
Suppose the function spaces $X,U,W^*$ and the  residual covariance operator $\sigres$ in \eqref{cov-res} satisfy
\begin{align}\label{embeddings}
W^*_\sigma \embed L^r(\Omega),\quad U\embed L^s(\Omega), \quad U\embed L^\frac{sr\gamma}{s-r}(\Omega), \quad B\in\Lc(X,\Lc(U,W^*_\sigma))\quad s\geq r%\blue{\geq 1. }  
\end{align}
with the Cameron-Martin subspace $W^*_\sigma=\sigres^{1/2}(W^*)$. Moreover, assume that the log-likelihood function $\Phi_\mathrm{AAO}$ is bounded at zero. Then Assumption \ref{assump} holds.  
\end{proposition}
\begin{remark}[a-b-c-s problem] \label{rem:abcs-problem}
    The framework in Proposition \ref{prop:wellposed} covers a large class of equations with a feasible choice of function space setting. To demonstrate this, we consider the  reaction-advection-diffusion equation
\begin{equation}\label{eq: abcs}
\begin{split}
    \nabla\cdot(a\nabla u)+\mathbf{b}\cdot\nabla u + cu +u^\gamma&=s \quad\text{on }\dom , \quad \gamma\geq0 \\
    u|_\bou&= 0 
    \end{split}
\end{equation}
with unknown diffusion $a$, advection $\mathbf{b}$, potential $c$ and source term $s$; so $\theta=(a,\b,c,s)\in X:= X_a\times X_b\times X_c\times X_s$. One could choose, for instance, with $d=3$,
\begin{align*}
&U=H^2(\dom)\cap H^1_0(\Omega)\embed L^\infty(\Omega),\quad X_a=H^2(\Omega), \quad X_b=H^1(\Omega),\quad  X_c=L^2(\Omega),\quad X_s=L^2(\Omega),\\ &W^*=H^{-2}(\Omega),\quad \sigres=(I-\Delta)^{-2}, \quad\text{and} \quad W^*_\sigma = L^2(\dom).
\end{align*}
or, with $d=2$,
\begin{align*}
&U=H^1_0(\Omega)\embed L^s(\Omega), s<\infty\quad X_a=H^2(\Omega), \quad X_b=H^1(\Omega),\quad  X_c=L^2(\Omega),\quad X_s=L^2(\Omega),\\ &W^*=H^{-2-\epsilon}(\Omega),\quad \sigres=(I-\Delta)^{-1-\epsilon}, \epsilon>0\quad\text{and} \quad W^*_\sigma = H^{-1}(\dom).
\end{align*}
\end{remark}
We are now ready to prove Proposition \ref{prop:wellposed}.

\begin{proof}[Proof of Proposition \ref{prop:wellposed}]

(i) Boundedness of $\llh$ from below is clear from its non-negativity.

(ii) Local boundedness of $\llh$ from above follows from its local Lipschitz continuity, which shall be proven in (iii)-(iv), and boundedness at the point $0\in X\times U\times Y$ via the splitting
$|\Phi_\mathrm{AAO}(\theta,u;y)|\leq |\Phi_\mathrm{AAO}(\theta,u;y)-\Phi_\mathrm{AAO}(0)|+|\Phi_\mathrm{AAO}(0)|$.

(iii) Local Lipschitz continuity in the unknowns: Consider $\|(\theta_1,  u_1)\|,\|(\theta_2,  u_2)\|,\|y\|<R$. Using the prallelogram $\|a\|^2-\|b\|^2=2(a-b,b)+\|a-b\|^2$, one estimates
\begin{align*}
  \Phi_\mathrm{AAO}(\theta_1,  u_1;y)&-\Phi_\mathrm{AAO}(\theta_2,  u_2;y)\\\
 &= \|y-L(u_1)\|^2_{\Sigma_\varepsilon}-\|y-L(u_2)\|^2_{\Sigma_\varepsilon}
 +\|F(\theta_1,u_1)\|^2_{\Sigma_{\text{res}}}-\|F(\theta_2,u_2)\|^2_{\Sigma_{\text{res}}}\\
 &=2\big(L(u_1)-L(u_2),L(u_2)-y\big)_\sigep+\|L(u_1)-L(u_2)\|_\sigep^2\\
 &\qquad + 2\big(F(\theta_1,u_1) -F(\theta_2,u_2), F(\theta_2,u_2)\big)_\sigres+ \|F(\theta_1,u_1)-F(\theta_2,u_2)\|^2_\sigres\\
 &=: M_1+M_2+N_1+N_2.
 \end{align*}
 
For the $M$ terms, since the measurement operator $L:U\to Y$ as in \eqref{L} is linear bounded, and $Y_\sigma$ is the finite-dimensional Cameron-Martin subspace $Y_\sigma:=\sigep^{1/2}(Y)$ %endowed with $\|\cdot\|_{Y_\sigma}:=\|\sigresvh\cdot\|_2$, 
with the positive definite covariance matrix $\sigep$ as in \eqref{eq:observation model2}, 
one has  $\|L(u_1-u_2)\|_{\sigep}\leq \|L\|_{U\to Y_\sigma}\|u_1-u_2\|_U$.
Hence
\begin{align}\label{lip-3-1}
|M_1+M_2|&\leq \|L\|_{U\to Y_\sigma}\|u_1-u_2\|_{U}\left(2R\|L\|_{U\to Y_\sigma}+2R+2R\|L\|_{U\to Y_\sigma}\right) \nonumber\\
&=:C(\|L\|,R)\|u_1-u_2\|_{U}.
\end{align}

Regarding the $N$ terms, we first evaluate $\|F(\theta_1,u_1)-F(\theta_2,u_2)\|_\sigres$.
We pay attention to the nonlinearity $f$ with  the bound on its growth assumed in \eqref{f-growth}. 
The fact that Cameron-Martin subspace $W^*_\sigma=\sigres^{1/2}(W^*)\subset W^*$ is a Sobolev space, thanks to the Mat\'ern-type covariance, allows us to estimate the $\sigres$-norm as Sobolev norm, then further employ Sobolev embeddings. In the following, we subsequently apply the growth condition \eqref{f-growth},  H\"older's inequality and the embeddings \eqref{embeddings} to obtain
\begin{align*}
 \|f(u_1)-f(u_2)\|_\sigres &\leq   C_{W^*_\sigma\to L^r}\|(1+|u_1|^\gamma +|u_2|^\gamma)|u_1-u_2|\|_{L^r}\\
 &\leq C_{W^*_\sigma\to L^r}(1+\|u_1\|^\gamma_{L^\frac{rs\gamma}{s-r}} +\|u_2\|^\gamma_{L^\frac{sr\gamma}{s-r}})\|u_1-u_2\|_{L^s}\\
 &\leq C_{W^*_\sigma\to L^r} C_{U\to L^s}C_{U\to L^\frac{sr\gamma}{s-r}}(1+\|u_1\|_U^\gamma+\|u_2\|_U^\gamma)\|u_1-u_2\|_U.
\end{align*}
Now, the term concerning the bilinear operator $B$ can be treated by inserting the auxiliary $[B\theta_1]u_2$, yielding $[B\theta_1]u_1- [B\theta_2]u_2 = [B\theta_1](u_1-u_2) + [B(\theta_1-\theta_2)]u_2$. This together with the above estimate for $f$ implies the total bound for $F$ as
\begin{align} \label{Lip-F}
\|F(\theta_1,u_1)&-F(\theta_2,u_2)\|_\sigres\leq \|[B\theta_1](u_1-u_2)\|_{W^*_\sigma} + \|[B(\theta_1-\theta_2)]u_2)\|_{W^*_\sigma} + \|f(u_1)-f(u_2)|\|_{W^*_\sigma} \nonumber\\
&\leq \|B\theta_1\|_{U\to W^*_\sigma}\|u_1-u_2\|_U +  \|B(\theta_1-\theta_2)\|_{U\to W^*_\sigma}\|u_2\|_U\nonumber \\
&\qquad+C_{W^*_\sigma\to L^r} C_{U\to L^s}C_{U\to L^\frac{sr\gamma}{s-r}}(1+\|u_1\|_U^\gamma+\|v_2\|_U^\gamma)\|u_1-u_2\|_U \nonumber\\
&\leq \|B\|_{X\to \Lc(U,W^*_\sigma)}\|\theta_1\|_X\|u_1-u_2\|_U + \|B\|_{X\to \Lc(U,W^*_\sigma)}\|\theta_1-\theta_2\|_X\|u_2\|_U \nonumber\\
&\qquad  +C_{W^*_\sigma\to L^r} C_{U\to L^s}C_{U\to L^\frac{sr\gamma}{s-r}}(1+\|u_1\|_U^\gamma+\|v_2\|_U^\gamma)\|u_1-u_2\|_U \nonumber\\
&\leq C(\|B\|,R,\gamma)\|(\theta_1,u_1)-(\theta_2,u_2)\|_{X\times U}
 \end{align}
with some positive constant $C$ depending on $R$, the differential operator $B$ and the nonlinearity growth $\gamma$.
Using this bound, we derive the estimate for the $N$ terms as
\begin{align}\label{lip-3-2}
  |N_1+N_2| &\leq \|F(\theta_1,u_1)-F(\theta_2,u_2)\|_\sigres\Big( 3\|F(\theta_2,u_2)\|_\sigres+\|F(\theta_1,u_1)\|_\sigres\Big) \nonumber\\
  &\leq  C(\|B\|,R,\gamma)\|(\theta_1,u_1)-(\theta_2,u_2)\|_{X\times U} 4\max_{i=1,2}\|F(\theta_i,u_i)\|_\sigres \nonumber\\
  &\leq C(\|B\|,\|F(0,0)\|,R,\gamma)\|(\theta_1,u_1)-(\theta_2,u_2)\|_{X\times U}.
\end{align}
Above evidently, the terms $\|F(\theta_i,u_i)\|, i=1,2$ are bounded by $\|F(\theta_i,u_i)-F(0,0)\|+\|F(0,0)\|_\sigres$, which is again by \eqref{Lip-F} bounded above by some $C(\|B\|,R,\gamma)$ and $\|F(0,0)\|$. This explains the constant $C(\|B\|,\|F(0,0)\|,R,\gamma)$ in the above estimate.
Finally, combining \eqref{lip-3-1} and \eqref{lip-3-2} proves local Lipschitz continuity \eqref{ass-lip-data} of $\Phi_\mathrm{AAO}$ w.r.t $(\theta,u)$.

(iv) Local Lipschitz continuity in data:  We now look at the difference of $\Phi_\mathrm{AAO}$ at any $\|y_1\|,\|y_2\|<R$. Similar to (iii), the parallelogram yields 
\begin{align*}
\Phi_\mathrm{AAO}(\theta,  u;y_1)-\Phi_\mathrm{AAO}(\theta,  u;y_2)
 &= \|y_1-L(u)\|^2_{\Sigma_\varepsilon}-\|y_2-L(u)\|^2_{\Sigma_\varepsilon}
% +\|F(\theta,u)-s_1\|^2_{\Sigma_{\text{res}}}-\|F(\theta,u)-s_2\|^2_{\Sigma_{\text{res}}} %skip s-term
\\
 &=2\big(y_1-y_2,y_2-Lu\big)_\sigep+\|y_1-y_2\|_\sigep^2. 
 %+ 2\big(s_1-s_2, s_2-F(\theta,u)\big)_\sigres+ \|s_1-s_2\|^2_\sigres %skip s-term
 \end{align*}
 Proceeding in the same fashion as \eqref{lip-3-1}-\eqref{lip-3-2}, with keeping in mind that the Lipschitz constant in \eqref{Lip-F} contains $\|u\|^\gamma$ explicitly, meaning that
 \begin{align*}
|\Phi_\mathrm{AAO}(\theta,  u;y_1)-\Phi_\mathrm{AAO}(\theta,  u;y_2)| 
&\leq (c_1\|u\|+c_2\|u\|^\gamma + c_3\|\theta\|+c_4R\big)\|y_1-y_2\|_{Y\times W^\ast}
 \end{align*}
 with some positive constants $c_1, c_2,c_3$.
The polynomial of order $\gamma$ in the Lipschitz constant is clearly dominated by the exponential rule required in \eqref{ass-lip-data} with some power $\beta>0$. Thus, the claim (iv) holds true. The proof is complete.
\end{proof}

\begin{remark}
In the reduced setting, Lipschitz continuity (or even differentiability) of the forward map relies on corresponding properties of the nonlinear parameter-to-state map, which in turn are typically established via the well-posedness of the linearized PDEs and the implicit function theorem \cite{TCC21,nguyen-passive}. In contrast, the all-at-once setting is more straightforward, as its forward map does not involve an inverse differential operator, as can be seen in the proof of Proposition~\ref{prop:wellposed}.
\end{remark}

With Assumption \ref{assump} established, we are now ready to show the well-posedness result.
\begin{theorem}[Well-posedness]\label{well-posedness Thm}
Let the probability measure $\mu_0$ be given by  $\mu_0(\theta, u) := \mu_{\theta_0}(\theta)\mu_{u_0}(u)$. 
Assume that the function $\Phi_\mathrm{AAO}\colon X \times U \times Y \to \R$ is as in \eqref{eq: log-likelihood for aao} and satisfies Assumptions~\eqref{assump}. If $\mu^y(d\theta,du)$ is the joint probability measure on $(\theta,u)$, given data $y$ and $F(\theta,u)=0$, then $\mu^y \ll \mu_0$ with Radon-Nikodym derivative 
\begin{equation}\label{eq:posterior-RN derivative}
    \frac{d\mu^y}{d\mu_0}(\theta,u)=\frac{1}{Z}\exp\big( -\Phi_\mathrm{AAO}(\theta,u;y) \big),
\end{equation}
where 
\begin{equation}\label{eq:normalization const-RN derivative}
    Z:=\int_{X\times U}\exp\big( -\Phi_\mathrm{AAO}(\theta,u;y) \big)\mu_0(d\theta,du)
\end{equation}
with $0<Z(y)<\infty$ for every $y$. Moreover, $\mu^y$ is locally Lipschitz with respect to data $y$ in the Hellinger distance; that is, there exists $C=C(r)>0$ such that
\begin{equation}\label{eq:hellinger}
    d_{\text{Hell}} (\mu^{y_1}, \mu^{y_2})\leq C \|y_1-y_2\|_Y
\end{equation}
for all $y_1,y_2\in Y$ with $\max\{\|y_1\|_Y, \|y_2\|_Y\}<r$.
\end{theorem}
\begin{proof} Proposition~\eqref{prop:wellposed} shows that $\Phi_\mathrm{AAO}$ satisfies Assumptions~\eqref{assump}.  
With this, the stability \eqref{eq:hellinger} then carries through
%over from the single-potential case 
following \cite[Theorem~4.2]{stuart2010inverse} or \cite[Theorem~14, Lemma~16]{DashtiStuart2017}).
    \end{proof}

\section{%Bayesian Approach for the 
Posterior Estimation in the All-at-Once Framework}

In this section, we propose two numerical approaches to solve the  Bayesian inversion in the AAO framework and give some practical sampling details. We shall construct and sample from the posterior distribution using a Metropolis-within-Gibbs (MwG) MCMC algorithm, which alternates between updating the unknown
state $u$ and the Bayesian model parameter $\theta$ (Algorithm \ref{alg:sampler}). In addition, we will give details on a conventional joint sampling method that updates $u$ and $\theta$ simultaneously (Algorithm \ref{alg:jointsampler}). Both approaches will be verified numerically on two PDE-based inverse problems later in Section~\ref{s:numerics}.

\subsection{Prior Sampling}\label{ss:prior sampling}
 For the discretization of a Gaussian random field prior in general, we employ the truncated Karhunèn-Loeve (KL) expansion.
The KL expansion represents a zero-mean Gaussian process 
$X(t, \omega)$ with a covariance operator $C$ as
\begin{equation}\label{eq: KL expansion}
    X(t, \omega) = \sum_{i=1}^{\infty} \sqrt{\lambda_i}\, \phi_i(t)\, \xi_i(\omega),
\end{equation}
where $\{\lambda_i, \phi_i\}_{i=1}^{\infty}$ are the eigenpairs of the covariance operator $C$ and $\xi_i \sim \mathcal{N}(0,1)$ are independent standard Gaussian random variables. The KL-expansion converges provided that the covariance $C$ is of trace class. 

For the Bayesian AAO, we define a joint prior for the parameter--state pair $(\theta,u)$ as the Gaussian measure
$$
\mu_0=\mathcal{N}(m_0,\Sigma_0),
$$
where $m_0:= (\theta_0,u_0)^T$, and $\Sigma_{0}:=\begin{pmatrix}
\Sigma_{\theta_0} &  0\\
0                 &  \Sigma_{u_0}   
\end{pmatrix}$
is a PDE-informed block-diagonal joint covariance. 
For the parameter prior, we assume symmetric, nonnegative, and trace-class prior covariance operators for $\Sigma_{\theta_0}\colon X\to X$. %and $\Sigma_{u_0}\colon U\to U$. 
Here, we take the Mat\'ern--Whittle covariance operator as %for the parameter $\theta$ defined by
\begin{align}\label{3.MW}
  \Sigma_{\theta_0}(x,y) = \sigma^2\frac{2^{1-\nu}}{\Gamma(\nu)}
  \bigg(\frac{d_0}{\ell}\bigg)^\nu K_\nu\bigg(\frac{d_0}{\ell}\bigg),
  \quad x,y\in \Omega,
\end{align}
where $d_0=\|x-y\|_2$, $\Gamma$ is the Gamma function, and $K_\nu$ is the modified Bessel function. The parameters $\sigma^2$, $\nu$, and $\ell$ denote the variance, smoothness, and spatial correlation length, respectively.

In order to construct a PDE-informed state  state covariance $\Sigma_{u_0}$, the parameter/source prior covariance $\Sigma_{\theta_0}$ is propagated through the linearized PDE solution map, i.e.,
\begin{align}\label{prior-u}
\Sigma_{u_0} := S'(\theta_0)\Sigma_{\theta_0} S'(\theta_0)^*,
\end{align} 
where $u_0=S(\theta_0)$ is the  prior mean of the state. 
Consequently, the resulting joint prior is a PDE-informed block-diagonal Gaussian prior, in the sense that the source component is endowed with a Matérn--Whittle covariance, whereas the state covariance incorporates the local sensitivity of the PDE solution to perturbations in the source.
We then define the joint prior measure as
\begin{equation}\label{eq:prior measure 2}
\mu_0
=
\mathcal{N}(\theta_0,\Sigma_{\theta_0})
\otimes
\mathcal{N}(u_0,\Sigma_{u_0})
=
\mathcal{N}(m_0,\Sigma_0).
\end{equation}

 We note that~\eqref{prior-u}  is an approximation, where the mean $u_0 := S(\theta_0)$ and covariance $\Sigma_{u_0}$ are derived from a linearization of the forward map $S$ around $\theta_0$. Although this linearization induces a nonzero prior cross-covariance
$
\operatorname{Cov}(\theta,u)
=
\Sigma_{\theta_0}S'(\theta_0)^*,
$
we neglect this term and employ the block-diagonal covariance $\Sigma_0$ in our formulation. 
This approximation is made to preserve the simple product-Gaussian Cameron--Martin structure and to avoid the additional complexity of constructing correlated proposals in Algorithm~\ref{alg:sampler}-\ref{alg:jointsampler}. As demonstrated in Section~\ref{s:numerics}, the resulting approximation has a negligible effect on the reconstruction.

\subsection{Covariances in the likelihood function}\label{ss: likelihood cov}

As described in Section \ref{ss:AAO-likelihood}, the likelihood function in the Bayesian AAO approach consists of two components: the data misfit with  the data covariance $\Sigma_\epsilon$ and the PDE residual with residual covariance $\Sigma_{\text{res}}$.
For $\Sigma_{\text{
res}}$, %the covariance corresponding to the PDE-residual component, 
we assume a  Mat\'ern--Whittle covariance operator as in \eqref{3.MW} with  the residual level $\epsilon_{\text{res}}^2$ and $s={\nu + d/2}$. %, according to \eqref{cov-res}.
For the data covariance, we use $\Sigma_\epsilon=\epsilon^2 I$ as stated in \eqref{eq:observation model2}. We will later numerically investigate the impact of the residual level on the proposed Bayesian AAO reconstruction approach (Section~\ref{s:numerics}).

\subsection{Posterior Sampling}

To sample from the posterior distribution $p(\theta, u \mid y)$ in an all-at-once fashion, we employ a hybrid MwG sampler \cite{robert_monte_2004}. The posterior is updated blockwise by alternating between the conditional distributions $p(\theta \mid u, y)$ and $p(u \mid \theta, y)$. This allows each block to be sampled using a proposal that is well suited to its structure, such as a preconditioned Crank--Nicolson (pCN)  proposal \cite{cotter2013mcmc} for function-valued quantities, and a Metropolis--Hastings proposal for finite-dimensional parameter vectors.

At each iteration, the sampler alternates between two conditional updates:
\begin{enumerate}
    \item The first step targets $p(\theta\mid  u, y)$. If $\theta$ is a finite-dimensional parameter vector, we use a Metropolis--Hastings update with proposal $\theta^* \sim \mathcal{N}(\theta_n, \sigma^2 I)$. When $\theta$ is a function, we use a pCN step which is well-posed in infinite dimensions and preserves the prior by construction. 
    \item The second step targets $p(u \mid \theta, y)$. Since $u$ is function-valued, we always employ a pCN step.
\end{enumerate} 
A detailed algorithmic description of this process is given in Algorithms~\ref{alg:sampler}.
Note that in the reduced setting, only step (1) is necessary to sample from $p(\theta | y)$.

\begin{algorithm}[h!]
\caption{Hybrid MwG MCMC sampler for the Bayesian AAO}
\label{alg:sampler}
\begin{algorithmic}[1]
    \State \textbf{Input:} Initial values $\theta_0$, $u_0$; pCN step-size $\beta \in (0,1]$; proposal variance $\sigma^2$; number of iterations $N$
    \State \textbf{Output:} Samples $\{(\theta_n, u_n)\}_{n=1}^{N}$
    \medskip
    \For{$n = 0, 1, \ldots, N-1$}
        \medskip
        \State \textbf{// Step 1: update for $\theta_{n+1} \mid u_n, y$}
        \If{Metropolis step}
            \State Draw proposal 
            \begin{equation*}
                \theta^* \sim \N\!\left(\theta_n,\; \sigma^2\, I\right)
            \end{equation*}
            \State Accept $\theta_{n+1} = \theta^*$ with probability $\min\!\left(1, \dfrac{p(\theta^* \mid u_n, y)}{p(\theta_n \mid u_n, y)}\right)$, otherwise set $\theta_{n+1} = \theta_n$
        \ElsIf{pCN step}
            \State Draw $\xi_n^\theta \sim p(\theta)$ and propose
            \begin{equation*}
                \theta^* = \sqrt{1 - \beta^2}\; \theta_n + \beta\, \xi_n^\theta
            \end{equation*}
            \State Accept $\theta_{n+1} = \theta^*$ with probability $\min\!\left(1, \dfrac{p(y \mid \theta^*, u_n)}{p(y \mid \theta_n, u_n)}\right)$, otherwise set $\theta_{n+1} = \theta_n$
        \EndIf
        \medskip
        \State \textbf{// Step 2: pCN update for $u_{n+1} \mid \theta_{n+1}, y$}
        \State Draw $\xi_n^u \sim p(u)$ and propose
        \begin{equation*}
            u^* = \sqrt{1 - \beta^2}\; u_n + \beta\, \xi_n^u
        \end{equation*}
        \State Accept $u_{n+1} = u^*$ with probability $\min\!\left(1, \dfrac{p(y \mid \theta_{n+1}, u^*)}{p(y \mid \theta_{n+1}, u_n)}\right)$, otherwise set $u_{n+1} = u_n$
        \medskip
    \EndFor
\end{algorithmic}
\end{algorithm}

Alternatively, it is also possible to update $(\theta, u)$ jointly as in Algorithm~\ref{alg:jointsampler}. Block-wise updates sometimes yield higher acceptance probabilities in practice because each proposal perturbs only one component of the state at a time. The trade-off is that when the parameter blocks are strongly correlated, MwG samplers may mix slowly, as single-block updates can struggle to traverse directions of strong posterior dependence. A comparison of alternating versus joint updates is examined in Section \ref{sec:joint-sampler}.

\begin{algorithm}[h!]
\caption{Joint MCMC sampler for the Bayesian AAO}
\label{alg:jointsampler}
\begin{algorithmic}[1]
    \State \textbf{Input:} Initial values $\theta_0$, $u_0$; pCN step-size $\beta \in (0,1]$; proposal variance $\sigma^2$; number of iterations $N$
    \State \textbf{Output:} Samples $\{(\theta_n, u_n)\}_{n=1}^{N}$
    \medskip
    \For{$n = 0, 1, \ldots, N-1$}
        \medskip
        \If{Metropolis step}
            \State Draw proposal 
            \begin{equation*}
                \theta^* \sim \N\!\left(\theta_n,\; \sigma^2\, I\right)
            \end{equation*}
        \ElsIf{pCN step}
            \State Draw $\xi_n^\theta \sim p(\theta)$ and propose
            \begin{equation*}
                \theta^* = \sqrt{1 - \beta^2}\; \theta_n + \beta\, \xi_n^\theta
            \end{equation*}
        \EndIf
        \State Draw $\xi_n^u \sim p(u)$ and propose
        \begin{equation*}
            u^* = \sqrt{1 - \beta^2}\; u_n + \beta\, \xi_n^u
        \end{equation*}
        \medskip
        \medskip
        \State Accept $( \theta_{n+1}, u_{n+1}) = (\theta^*, u^*)$ with probability $\min\!\left(1, \dfrac{p(y \mid \theta^*, u^*)}{p(y \mid \theta_{n+1}, u_n)}\right)$, 
        \State otherwise set $(\theta_{n+1}, u_{n+1}) = (\theta_n, u_n)$.
        \medskip
    \EndFor
\end{algorithmic}
\end{algorithm}

\section{Numerical Experiments and Application To Semiconductor Devices}\label{s:numerics}
In this section, we consider numerical implementations of the AAO approach for the protptype a-b-c-s problem (Section \ref{num:abcs-problem}) then the semiconductor application (Section \ref{sec:semiconductor}).  For the  a-b-c-s problem, we reconstruct the unknown finite-dimensional parameters $a,b,c,s$, while for the semiconductor problem, the unknown doping profile $C \in H^1(\Omega)$ is a piecewise smooth function defined on two connected subdomains of $\dom$. 
Placing more emphasis on the semiconductor application, our plan is to examine the effect of
\begin{itemize}
    \item a baseline versus tightened prior for $u$ (Sections \ref{sec:prior-u-standard}-\ref{sec:prior-u-inform})
    
    \item PDE residual level (Section \ref{sec:res-level})
\item alternating versus jointly update $\theta, u$ (Section \ref{sec:joint-sampler})
\end{itemize}
on the reconstruction.

 We implement the numerical examples with the finite element software NGSolve~\cite{Schoeberl2014}. In particular, the state $u\in H^1(\dom)$ is discretized using finite elements  with polynomials of degree one. The doping profile $C$ is discretized using the same finite element space. 
 We use a mesh size of $0.05$ for the a-b-c-s problem, and of $0.1$ for the semiconductor problem with local refinement around the junction $\Gamma$ to resolve the discontinuity of $C$.

\subsection{The a-b-c-s problem}\label{num:abcs-problem}
We consider the a-b-c-s problem discussed in Remark \ref{rem:abcs-problem} 
\begin{align*}
    \nabla\cdot(a\nabla u)+\mathbf{b}\cdot\nabla u + cu +u^\gamma&=s \quad\text{on }\dom , \quad \gamma=2 \\
    u|_\bou&= 0
\end{align*}
 on the 2D domain $\Omega:=[0, 1]^2$. The parameters are $\theta: = (a, \mathbf{b}, c, s)$ where $a, c, s \in \R$ and $\b \in \R^2$. The true value of $\theta$ is given by drawing from a zero--centered Gaussian with identity covariance $\theta \sim \N\left((1, 1, 1, -1, 1)^\top, I_5\right)$, and the prior $p(\theta)$ is a Gaussian with identity covariance centered around a perturbation of the true values. 
The prior $p(u)$ is described in Section \ref{ss:prior sampling}.

The observation operator $L$ takes a solution $u$ to \eqref{eq: abcs} and maps it to the values of $u$ at an equidistant grid of $[0, 1]^2$ with $m = 20^2$ observed values. Observational noise follows the distribution $\N(0, \delta^2 I)$ with noise level $\delta = 0.005$, which is roughly $5\%$ of the maximum amplitude of $u$. For posterior sampling, we set the step size $\beta = 0.2$, the proposal variance $\sigma^2= 0.2^2$ and choose the residual covariance $\Sigma_{\text{res}}$ as in \eqref{3.MW} 
with variance $\sigma^2 = 0.05^2$, length-scale $l = 0.1$ and smoothness $\nu = 1$.

Figure~\ref{fig: abcs hist} depicts the posterior means with a confidence band and ground truth values.
For  demonstration, $2,000$ posterior samples for the AAO and reduced setting were generated. The first $1,000$ samples are discarded as burn-in, and the remaining samples are thinned by the factor $5$, meaning that four out of five consecutive samples are discarded to reduce correlations. Noticeably, generating posterior samples takes  $\approx6$ minutes for the reduced setting and approximately $\approx 3$ minutes for the AAO setting. This time reduction illustration the computational efficiency of the AAO inversion.

\begin{figure}[h!]
    \centering
    \includegraphics[width=0.5\linewidth]{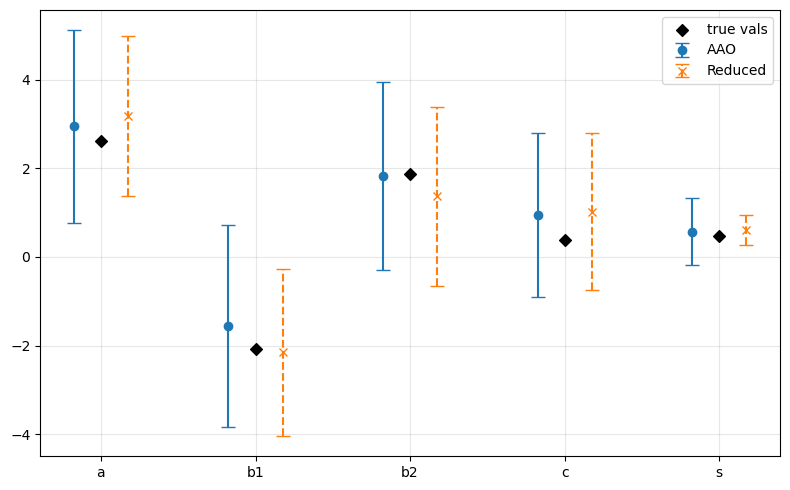}
    \caption{The a-b-c-s problem. Posterior means for AAO (blue) and reduced setting (orange) of four parameters $(a,(b_1,b_2),c,s)$ with a confidence band based on three posterior standard deviations and ground truth values (black). The variance for AAO appears slightly larger than the reduced setting, while posterior means of AAO are generally  closer to the ground truths.}
    \label{fig: abcs hist}
\end{figure}

In Figure~\ref{fig:abcs us}, the true state $u$, 
the posterior mean using the AAO approach, and a sample for the prior $p(u)$ are depicted. We observe that the variance for AAO inversion appears slightly larger than the reduced setting, while posterior means of the AAO approach
are generally slightly closer to the ground truths.

\begin{figure}[htbp]
\centering

\begin{subfigure}{0.32\textwidth}
    \centering
    \includegraphics[width=\linewidth]{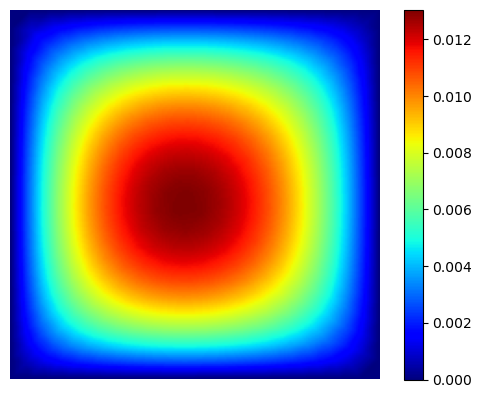}
    \caption*{true $u$}
\end{subfigure}
\hfill
\begin{subfigure}{0.32\textwidth}
    \centering
    \includegraphics[width=\linewidth]{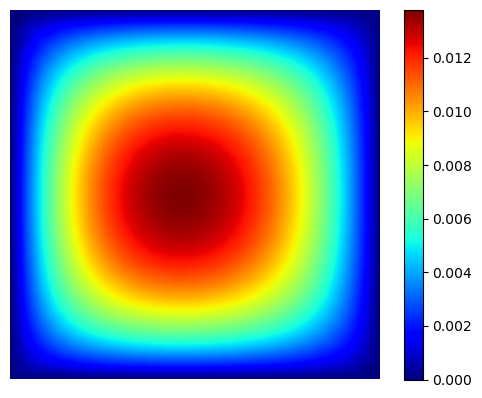}
    \caption*{mean of $u$ in AAO}
\end{subfigure}
\hfill
\begin{subfigure}{0.32\textwidth}
    \centering
    \includegraphics[width=\linewidth]{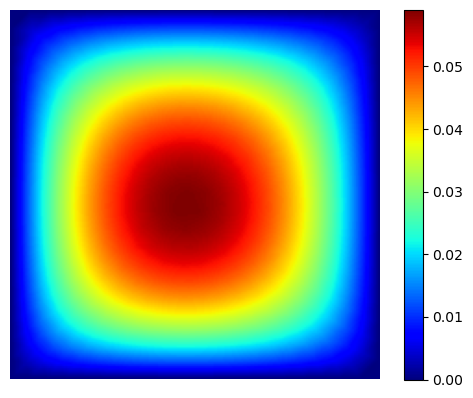}
    \caption*{sample from prior on $u$}
\end{subfigure}
\caption{a-b-c-s problem. Left: true state $u$ (corresponding to true $\theta$); middle: posterior mean of $u$; right: a prior sample from $p(u)$.}
\label{fig:abcs us}
\end{figure}

\FloatBarrier

\subsection{An application in semiconductor devices} \label{sec:semiconductor}
In this section, we illustrate the performance of the proposed AAO approach on a benchmark problem from semiconductor device modeling %, namely the Bayesian posterior reconstruction of 
for the doping profile in a 
$pn$-diode. Under a simplified stationary drift-diffusion model \cite{burger2001identification,taghizadeh2025bayesian}, the electrostatic potential 
$u$ inside the device is governed by a semilinear elliptic PDE, and the unknown doping profile 
$C=C(x)$ enters as the source term, coupling the $p$- and $n$-doped regions across a fixed junction; see Figure \ref{fig-semiconductor} for illustration. Reconstructing 
$C$ from noisy and indirect boundary observations of 
$u$ represents a modern and important application of the Bayesian inverse problem constrained by PDE. We begin with describing the forward model and the prior assumptions on the doping profile, and then examining the performance of the proposed AAO inversion approach compared to the conventional reduced method.
%%%%%%%%%%%%%%%%%%%%%%%%%%
\usetikzlibrary{patterns,calc}
%%%%%%%%%%%%%%%%%%%%%%%%%%%
\begin{center}\label{fig-semiconductor}
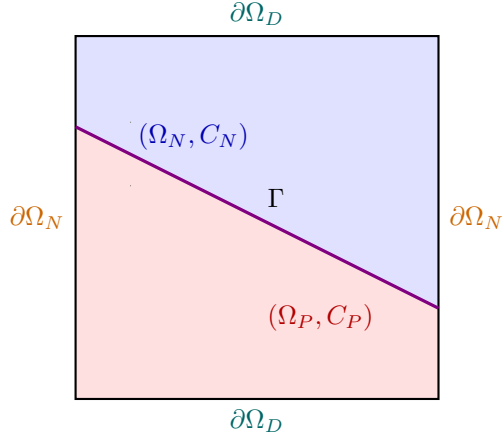
\begin{figure}
\begin{tikzpicture}[scale=1.2]
  % --- Geometry ---
  \coordinate (BL) at (0,0); % bottom-left
  \coordinate (BR) at (4,0); % bottom-right
  \coordinate (TR) at (4,4); % top-right
  \coordinate (TL) at (0,4); % top-left
  \coordinate (Jleft)  at (0,3);
  \coordinate (Jright) at (4,1);

  % --- Colored regions ---
  \begin{scope}
    \clip (BL) -- (BR) -- (TR) -- (TL) -- cycle;
    \fill[blue!12] (Jleft) -- (TL) -- (TR) -- (Jright) -- cycle;   % Omega_N region
    \fill[red!12]  (Jleft) -- (Jright) -- (BR) -- (BL) -- cycle;   % Omega_P region
  \end{scope}

  % --- Outer boundary ---
  \draw[thick] (BL) -- (BR) -- (TR) -- (TL) -- cycle;

  % --- Junction line ---
  \draw[very thick, violet] (Jleft) -- (Jright);
  \node[font=\footnotesize, above right] at ($(Jleft)!0.5!(Jright)$) {$\Gamma$};

  % --- Region labels ---
  \node[font=\footnotesize, blue!70!black] at (1.3,2.9) {$(\Omega_N, C_N)$};
  \node[font=\footnotesize, red!70!black]  at (2.7,0.9) {$(\Omega_P, C_P)$};

  % --- Boundary condition labels ---
  \node[font=\footnotesize, above, teal!80!black] at (2,4) {$\partial\Omega_D$};
  \node[font=\footnotesize, below, teal!80!black] at (2,0) {$\partial\Omega_D$};
  \node[font=\footnotesize, left,  orange!80!black]  at (0,2) {$\partial\Omega_N$};
  \node[font=\footnotesize, right, orange!80!black] at (4,2) {$\partial\Omega_N$};

  % --- Depletion region hint ---
  \draw[dashed, gray] ($(Jleft)!0.15!(Jright) + (0,0.35)$) -- ($(Jright)!0.85!(Jleft) + (0,0.35)$);
  \draw[dashed, gray] ($(Jleft)!0.15!(Jright) + (0,-0.35)$) -- ($(Jright)!0.85!(Jleft) + (0,-0.35)$);
\end{tikzpicture}
\caption{Illustration of the semiconductor domain $\Omega = \Omega_N\cup \Omega_P \subset \R^2 $ featuring $N$-type and $P$-type regions with  junction $\Gamma$ and boundary $\partial \Omega =  \partial \Omega_D \cup  \partial \Omega_N$.}
\end{figure}
\end{center}
%%%%%%%%%%%%%%%%%%%%%%%%%%
Our goal is to determine $C$ in elliptic PDE of the form
\begin{align}\label{eq:sensor model}
\begin{aligned}
    -\Delta u + f(u) &= C, \qquad &&x\in \Omega\subset \R^2 \\
    u &= u_{bi}, \qquad &&x \in \partial \Omega_D \\
    \nabla u \cdot \nu &= 0 \qquad && x \in \partial \Omega_N, 
\end{aligned}
\end{align}
modeling electrostatic potential in $pn$-junction diodes (c.f.~\cite{taghizadeh2025bayesian,schneider2026convex}),  where $\nu$ is the outer normal vector on $\partial \Omega_N$, and $u_{bi}$ a small built-in voltage and with $f(u) =\tfrac12 (e^u-e^{-u})$. The domain $\Omega = [-1, 1]^2$ models the cross-section of a $pn$-diode with Ohmic contacts $\Gamma_N = [-1, 1] \times \{+1\}$ and $\Gamma_P = [-1, 1] \times \{-1\}$, corresponding to two subdomains $\Omega_N$ and $\Omega_P$ connected by the junction line $\Gamma$.
We assume that the doping profile $C$ is given by
\begin{eqnarray}\label{eq: piecewise constant C}
    C(x) = \begin{cases}
        C_N(x)  & x \in \Omega_N, \\
        C_P(x) & x \in \Omega_P,
    \end{cases}
\end{eqnarray}
where $C_N$ and $C_P$ are independent Gaussian random fields. Both fields are equipped with a Mat\'{e}rn--Whittle covariance kernel with variance $\sigma^2 = 0.8^2$, smoothness parameter $\nu = 1.5$, and spatial correlation length $l = 4$.  

{The measurement is given by an observation of the boundary flux $\nabla u \cdot \nu$ on an equidistant grid at the top part of the boundary, i.e. $\partial\Omega_D \cap \overline{\Omega_N}$}
% reiterate that we impose Dirichlet conditions on this part of the boundary and a measurement of the boundary flux is informative.
%-- \red{How so?  In the eq. $\nabla u \cdot \nu=0$.}{\blue{[yes, but only on the right and left sides. We measure on the top boundary, where this condition does not apply.]}} 
and observational noise is given by a centered Gaussian distribution with covariance $\delta^2 I$ with noise level $\delta = 0.05$. The PDE residual is a Mat\'ern-Whittle covariance with variance $\sigma^2 = 0.05^2$, length-scale $l = 0.1$ and smoothness $\nu = 1$.

We consider two choices of prior for $u$.
\begin{enumerate}
    \item[(i)]\label{item:baseline-prior} \textbf{Baseline prior.} As the baseline prior, we use the prior $\mathcal{N}(u_0, \Sigma_{u_0})$ with covariance $\Sigma_{u_0}$ constructed via the linearized PDE solution map as in~\eqref{prior-u}. This represents the baseline level of prior knowledge on $u$, encoding only the uncertainty inherited from the parameter prior $\Sigma_{\theta_0}$ through the linearized forward map.
    
    \item[(ii)]\label{item:tightened-prior} \textbf{Tightened prior.} To investigate the sensitivity of the AAO approach to the prior specification on $u$, we consider a more concentrated prior $\mathcal{N}(u_0, 0.5^2\,\Sigma_{u_0})$. The covariance is shrunk with a factor of $0.5^2$ compared to the choice (i). This represents a scenario in which the practitioner assigns higher prior confidence to the neighbourhood of the linearized state $u_0 = S(\theta_0)$. We examine how the tightness of the state prior affects posterior accuracy and computational efficiency.
\end{enumerate}
% \red{I think we need to name the priors. From the descriptions, both covariances $\Sigma_{u_0}$ are constructed via the linearized PDE solution map in~\eqref{prior-u} (which is good); the difference is just the scaling factor regarding confidence. Naming one 'standard' and the other 'informed' might cause confusion. I think we should adjust the names here and others places they appear as well as in the abstract and conclusion.}
% \fabian{I renamed 'standard' to 'baseline' and 'informed' to 'tightened'. There is still some occurrences of 'PDE-informed prior' which describe how we put a prior on $u$}

\subsubsection{Baseline prior.}\label{sec:prior-u-standard} We use the baseline prior described in item \ref{item:baseline-prior}. 
We generate 15,000 samples, which are then thinned to 1,500 samples from the posterior distribution in the AAO and the reduced setting. 
As can be seen in Figure~\ref{fig: semiconductor; Cmean} and Table~\ref{tab: AAO metrics}, both approaches exhibit comparable MSE and empirical means. Overall, the variance in the AAO settings is slightly larger and divided more evenly in the domain. 

Importantly, generating samples with AAO setting is approximately twice as fast compared to the reduced setting. We therefore generate 30,000 samples, thinned to 3,000, for the AAO setting, %in comparable computational time to 
while generating 15,000 samples for the reduced setting. The larger sample count further reduces the bias, demonstrating that the AAO approach offers a dual advantage: one may either match the reduced setting's sample count and achieve a comparable posterior estimate in half the time, or exploit the speed gain to draw more samples and obtain a more 
accurate estimate within the same budget.

\begin{figure}[h]
\begin{subfigure}{0.23\textwidth}
    \centering
    \includegraphics[width=\linewidth]{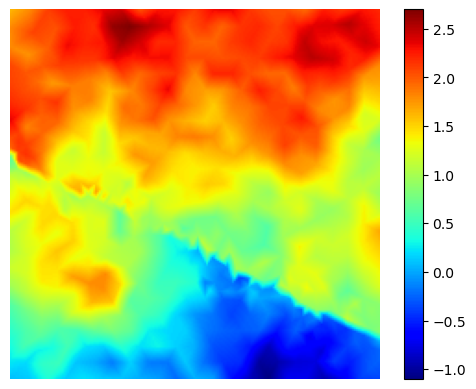}
    \caption*{ground truth C}
\end{subfigure}
\begin{subfigure}{0.23\textwidth}
    \centering
    \caption*{\textbf{Reduced}, 15k samples}
    \includegraphics[width=\linewidth]{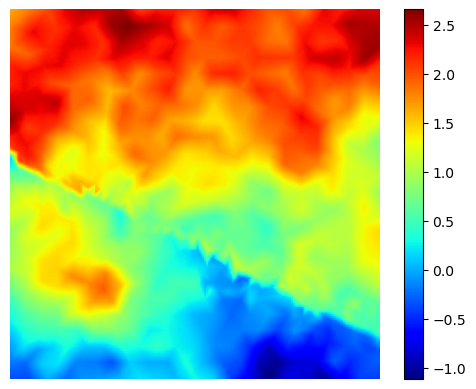}
    \caption*{posterior mean}
\end{subfigure}
\begin{subfigure}{0.23\textwidth}
    \centering
    \caption*{\textbf{AAO}, 15k samples}
    \includegraphics[width=\linewidth]{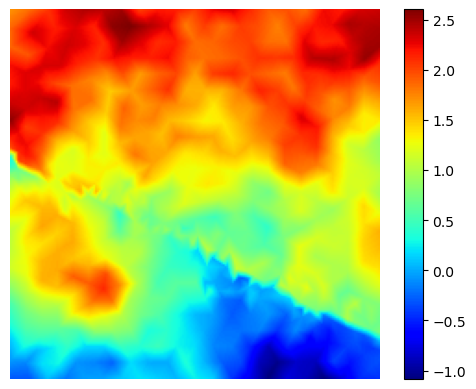}
    \caption*{posterior mean}
\end{subfigure}
\begin{subfigure}{0.23\textwidth}
    \centering
    \caption*{\textbf{AAO}, 30k samples}
    \includegraphics[width=\linewidth]{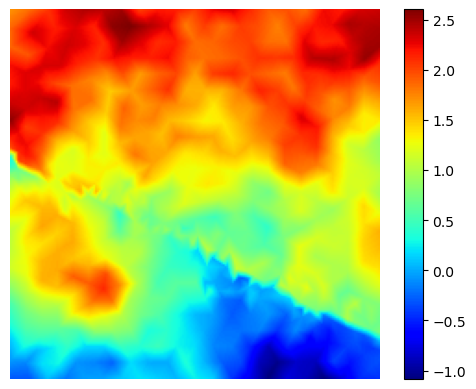}
    \caption*{posterior mean}
\end{subfigure}\\
\begin{subfigure}{0.23\textwidth}
    \phantom{\includegraphics[width=\linewidth]{images/semiconductor/Ctrue.png}}
\end{subfigure}
\begin{subfigure}{0.23\textwidth}
    \centering
    \includegraphics[width=\linewidth]{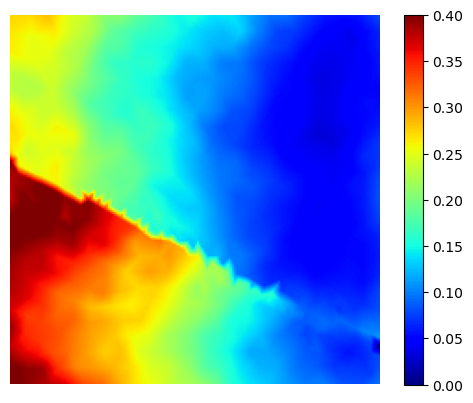}
    \caption*{variance}
\end{subfigure}
\begin{subfigure}{0.23\textwidth}
    \centering
    \includegraphics[width=\linewidth]{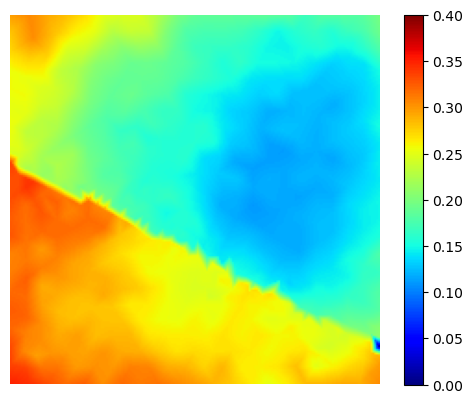}
    \caption*{variance}
\end{subfigure} 
\begin{subfigure}{0.23\textwidth}
    \centering
    \includegraphics[width=\linewidth]{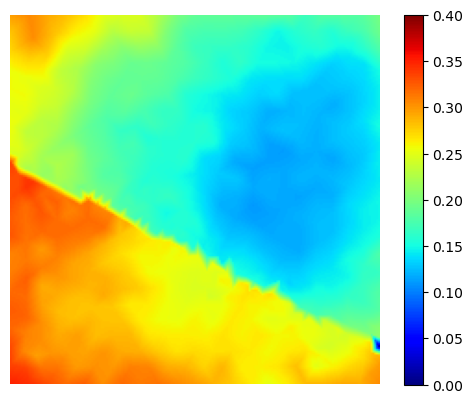}
    \caption*{variance}
\end{subfigure} \\
\begin{subfigure}{0.23\textwidth}
    \phantom{\includegraphics[width=\linewidth]{images/semiconductor/Ctrue.png}}
\end{subfigure}
\begin{subfigure}{0.23\textwidth}
    \centering
    \includegraphics[width=\linewidth]{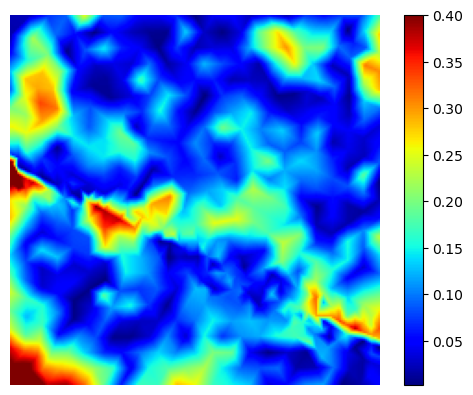}
    \caption*{bias}
\end{subfigure}
\begin{subfigure}{0.23\textwidth}
    \centering
    \includegraphics[width=\linewidth]{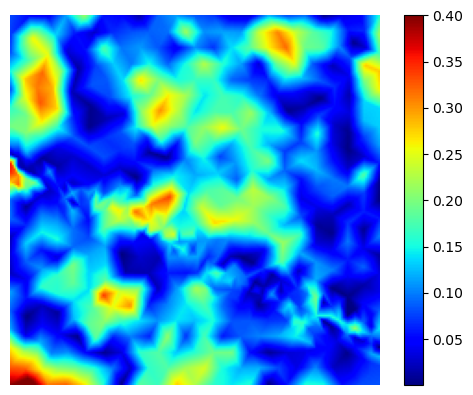}
    \caption*{bias}
\end{subfigure}
\begin{subfigure}{0.23\textwidth}
    \centering
    \includegraphics[width=\linewidth]{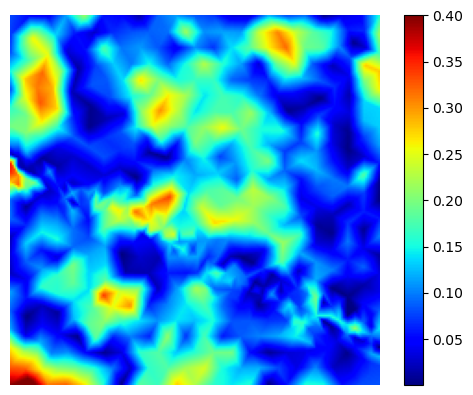}
    \caption*{bias}
\end{subfigure}
\caption{Posterior statistics of the doping profile $C$ for the AAO with baseline prior and reduced setting using $15,000$ posterior samples for both and additionally $30,000$ samples for AAO. \emph{Top row}: true $C$ and posterior means. \emph{Middle row}: posterior variances. \emph{Bottom row}: pixel--wise bias. The AAO approach achieves similar bias in about half the computational time compared to the reduced approach.}
\label{fig: semiconductor; Cmean}
\end{figure}

\begin{table}[h]
    \centering
    \caption{Comparison of computational time (in minutes), bias, variance, MSE and acceptance rate for the AAO with baseline prior and the classical reduced approach. The AAO approach achieves similar bias in about half the computational time and lower bias for comparable computational time versus the reduced approach.}
    \label{tab: AAO metrics}
    \begin{tabular}{c c c c c c c c}
        \hline
        & samples & time & bias & variance & MSE & \multicolumn{2}{c}{acceptance rate} \\
        \cline{7-8}
        & & & & & & $\theta$ & $u$ \\
        \hline
        Reduced & 15,000 & 17 & 0.0246 & 0.1736 & 0.1982 & 18.8\% & -- \\
        AAO & 15,000 & 9 & 0.0238 & 0.1990 & 0.2228 & 70\% & 12.4\% \\
        AAO & 30,000 & 18 & 0.0197 & 0.2134 & 0.2332 & 70\% & 12.3\% \\
        \hline
    \end{tabular}
\end{table}

\subsubsection{Tightened prior}\label{sec:prior-u-inform}
We consider a tightened, more informative prior by defining the prior on $u$ centered on $S(\theta_0)$ and with covariance scaled by $0.5^2$ relative to the baseline prior, as described in item \ref{item:tightened-prior}. 
Table~\ref{tab:  metrics 2} and Figure~\ref{fig: semiconductor; informative prior and joint sampler} compares the posterior results of the two prior scenarios, baseline versus tightened prior. We observe that, as expected, the tightened prior leads to a decrease in computational time, accelerating posterior sampling by a factor of about $1.6$ due to the higher acceptance rate of $u$. In addition, the posterior variance as well as the mean squared error (MSE) of the Bayesian AAO with the proposed tightened prior have decreased slightly.

\begin{table}[h]
    \centering
    \caption{Comparison of computational time (in minutes), bias, variance, MSE and acceptance rate for the AAO approach with the baseline and tightened prior. The tightened prior achieves higher acceptance rate and lower computational time.}
    \label{tab:  metrics 2}
\begin{tabular}{c c c c c c c c c}
     &  prior & samples & time & bias  & variance & MSE  & acceptance $\theta$ & acceptance $u$\\
     \hline
     AAO  & baseline & 30,000 & 18 & 0.0197 & 0.2134 & 0.2332 & 70\% & 12.3\% \\
    AAO  & tightened & 30,000 & 11 & 0.0207 & 0.1702 & 0.1908 & 70\% & 31.5\% \\
    \end{tabular}
\end{table}

\subsubsection{Influence of the PDE residual level %on the Bayesian posterior reconstruction
}\label{sec:res-level}
In this section, we analyze the influence of the PDE residual level on the posterior samples. The residual covariance $\sigres$ as described is chosen by Mat\'ern-Whittle covariance with varying variance $\sigma^2$, constant length-scale $l = 0.1$ and smoothness $\nu = 1$.
In Figure~\ref{fig:aao_delta}, we find that the MSE, variance and bias converge to the level of the reduced approach as $\sigma^2$ decreases, whereas the computational time blows up. Nonetheless, there is a regime of medium $\sigma^2$, where the AAO approach achieves faster posterior sampling.

\begin{figure}[h!]
    \centering
    \begin{subfigure}[t]{0.4\textwidth}
        \centering
        \includegraphics[width=\textwidth]{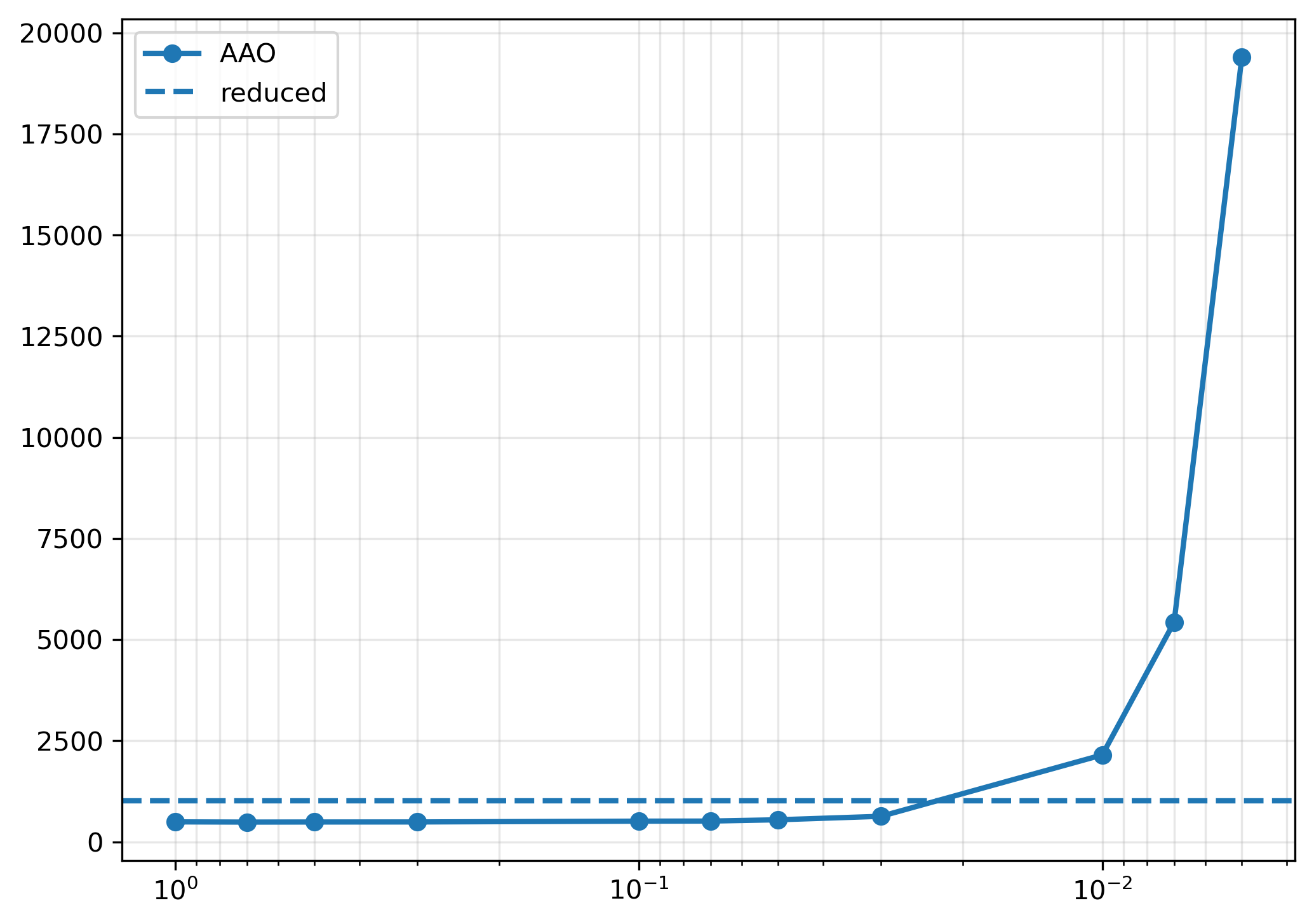}
        \caption*{computational time}
    \end{subfigure}
    %\hfill
    \begin{subfigure}[t]{0.4\textwidth}
        \centering
        \includegraphics[width=\textwidth]{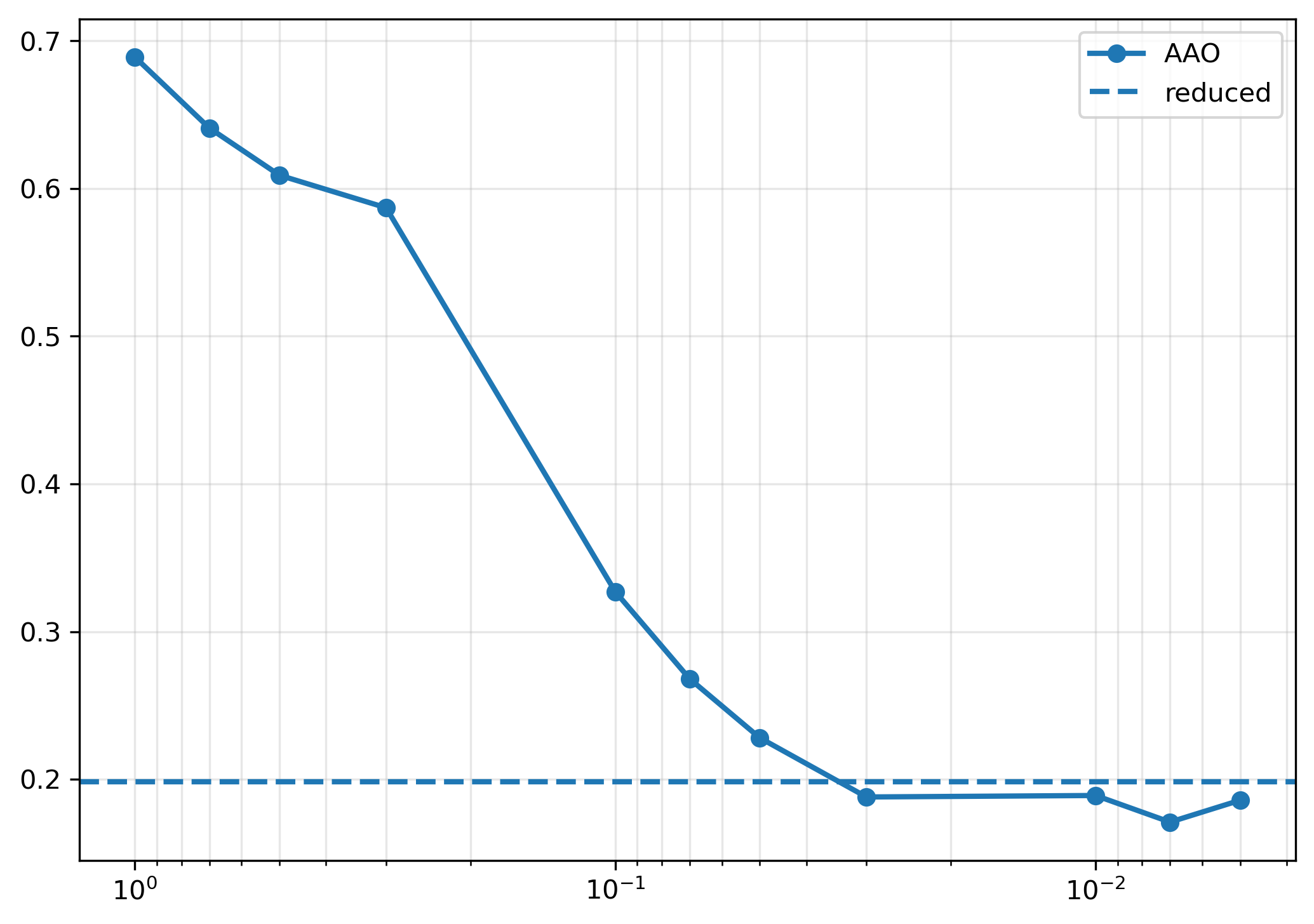}
        \caption*{MSE}
    \end{subfigure}
    %\hfill
    \begin{subfigure}[t]{0.4\textwidth}
        \centering
        \includegraphics[width=\textwidth]{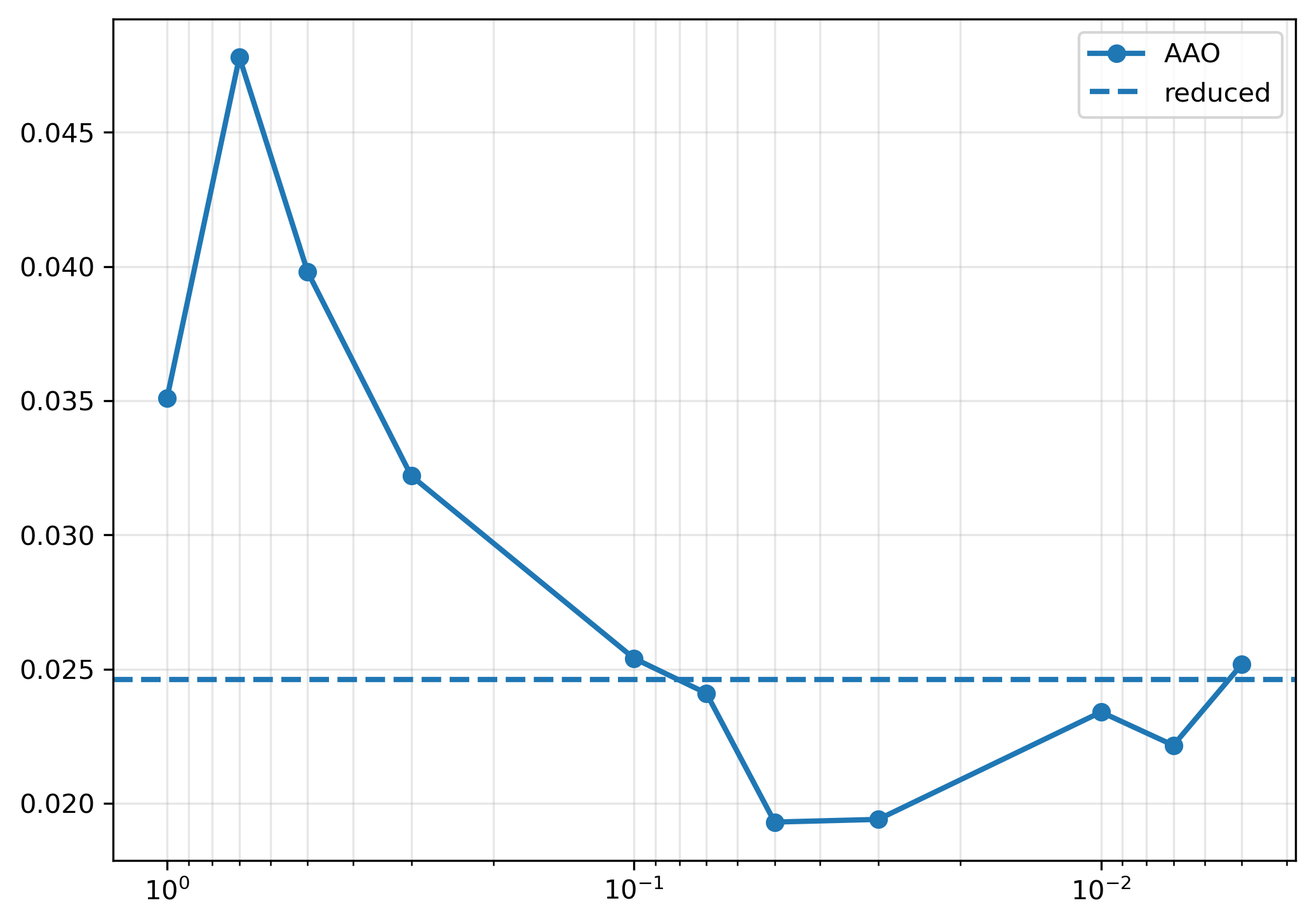}
        \caption*{bias}
    \end{subfigure}
    %\hfill
    \begin{subfigure}[t]{0.4\textwidth}
        \centering
        \includegraphics[width=\textwidth]{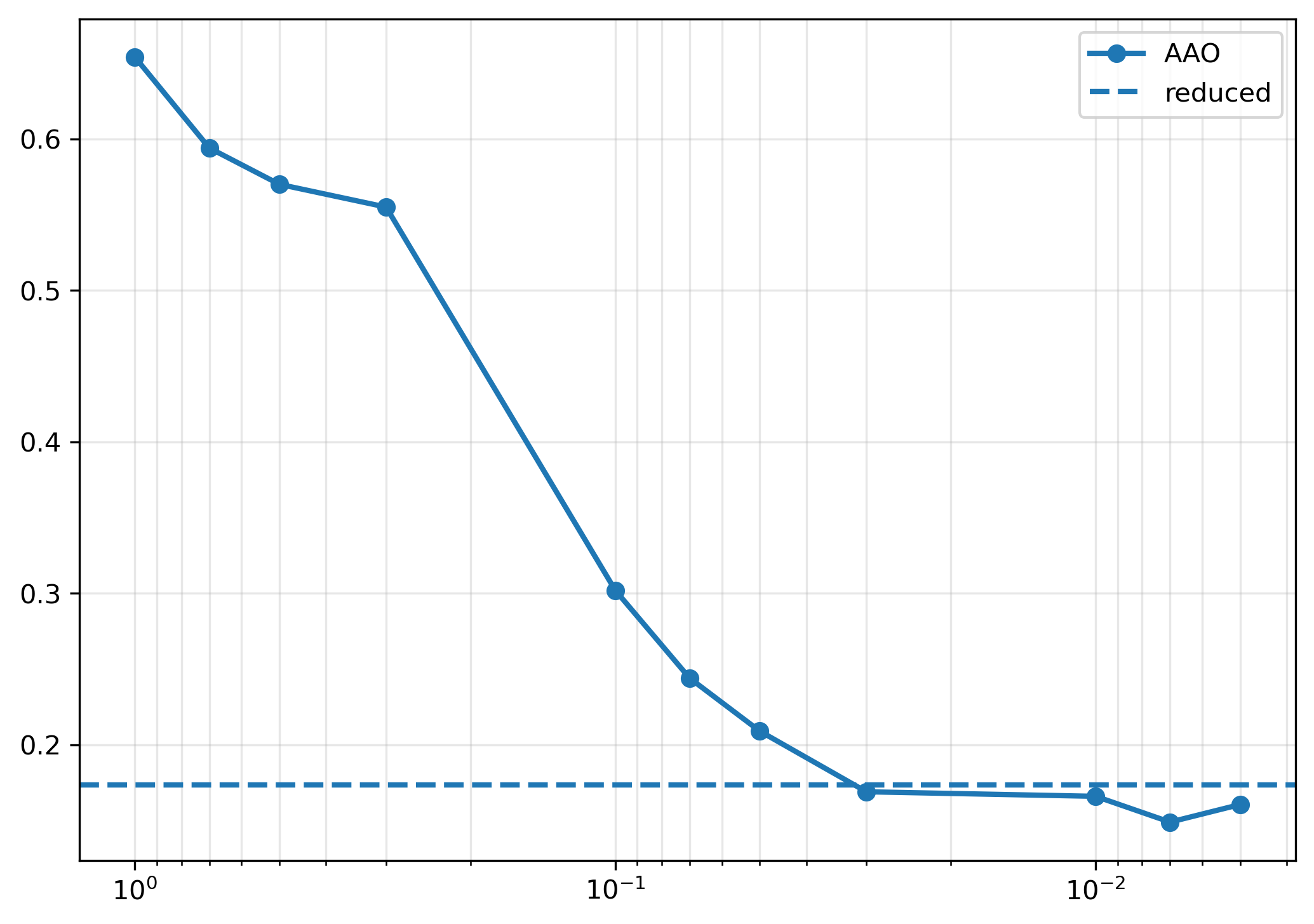}
        \caption*{variance}
    \end{subfigure}
    \caption{Performance of AAO approach with baseline prior for different variance level $\sigma^2$ in the residual covariance $\Sigma_{res}$. %of Mat\'ern-Whittle type with fixed length-scale and smoothens as a function of $\sigma$ 
    Horizontal axis in all plots: $\sigma$. Vertical axis in top left: computational time in seconds,  in top right: MSE, in bottom left: bias, in bottom right: variance.
    The dashed horizontal lines indicate the corresponding reduced values. MSE, variance and bias converge to the level of the reduced approach as $\sigma^2$ decreases.}
    \label{fig:aao_delta}
\end{figure}

\subsubsection{Joint sampler} \label{sec:joint-sampler}
Finally, we compare the proposed MwG sampler Algorithm \ref{alg:sampler} with the joint sampling Algorithm \ref{alg:jointsampler}. 
In Table \ref{tab:  joint sampler}, the joint sampling algorithm is able to slightly decrease posterior variance with lower computational time at the cost of a small increase in bias. 
Additionally, Figure~\ref{fig: semiconductor; informative prior and joint sampler} shows that MwG and joint samplers achieve similar posterior mean and variance. The tightened prior helps to reduce posterior variance effectively.

\begin{figure}[h]
\begin{subfigure}{0.23\textwidth}
    \centering
    \includegraphics[width=\linewidth]{images/semiconductor/Ctrue.png}
    \caption*{ground truth C}
\end{subfigure}
\begin{subfigure}{0.23\textwidth}
    \centering
    \caption*{\textbf{AAO}, baseline prior and MwG sampler}
    \includegraphics[width=\linewidth]{images/semiconductor/CpostAAO30k.png}
    \caption*{posterior mean}
\end{subfigure}
\begin{subfigure}{0.23\textwidth}
    \centering
    \caption*{\textbf{AAO}, baseline prior and joint sampler}
    \includegraphics[width=\linewidth]{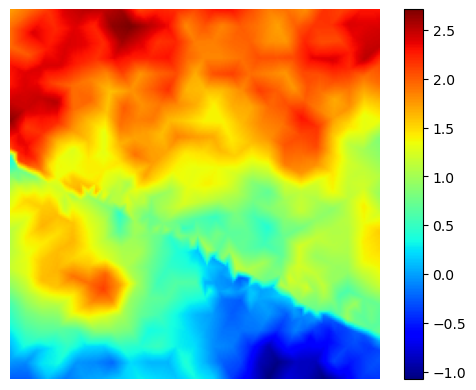}
    \caption*{posterior mean}
\end{subfigure}
\begin{subfigure}{0.23\textwidth}
    \centering
    \caption*{\textbf{AAO}, tightened prior and MwG sampler}
    \includegraphics[width=\linewidth]{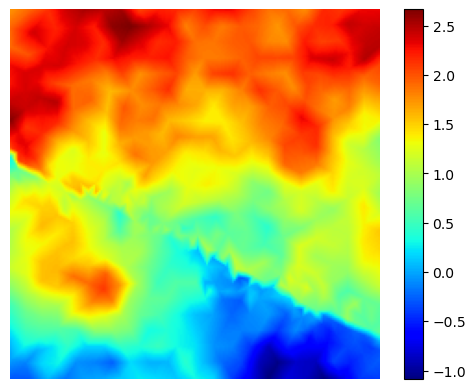}
    \caption*{posterior mean}
\end{subfigure}
\\
\begin{subfigure}{0.23\textwidth}
    \phantom{\includegraphics[width=\linewidth]{images/semiconductor/Ctrue.png}}
\end{subfigure}
\begin{subfigure}{0.23\textwidth}
    \centering
    \includegraphics[width=\linewidth]{images/semiconductor/VarAAO30k.png}
    \caption*{variance}
\end{subfigure}
\begin{subfigure}{0.23\textwidth}
    \centering
    \includegraphics[width=\linewidth]{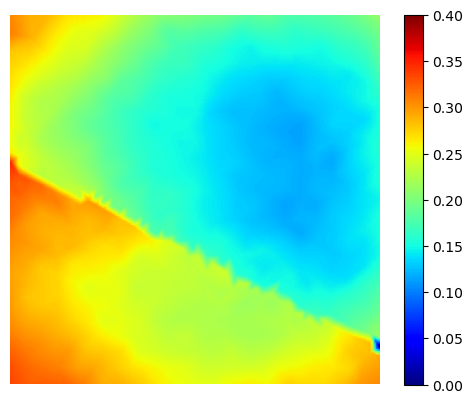}
    \caption*{variance}
\end{subfigure} 
\begin{subfigure}{0.23\textwidth}
    \centering
    \includegraphics[width=\linewidth]{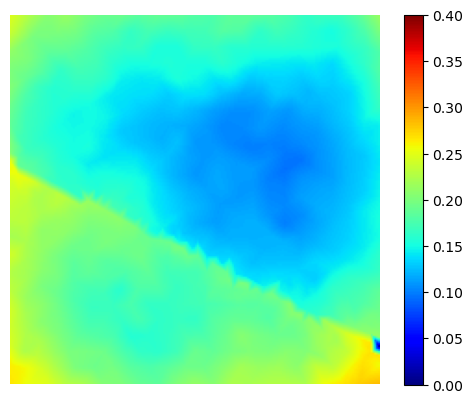}
    \caption*{variance}
\end{subfigure} 
\\
\begin{subfigure}{0.23\textwidth}
    \phantom{\includegraphics[width=\linewidth]{images/semiconductor/Ctrue.png}}
\end{subfigure}
\begin{subfigure}{0.23\textwidth}
    \centering
    \includegraphics[width=\linewidth]{images/semiconductor/BiasAAO30k.png}
    \caption*{bias}
\end{subfigure}
\begin{subfigure}{0.23\textwidth}
    \centering
    \includegraphics[width=\linewidth]{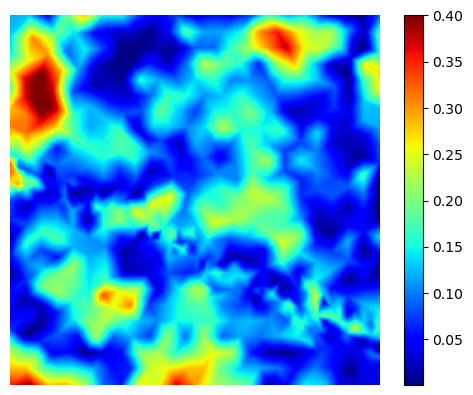}
    \caption*{bias}
\end{subfigure}
\begin{subfigure}{0.23\textwidth}
    \centering
    \includegraphics[width=\linewidth]{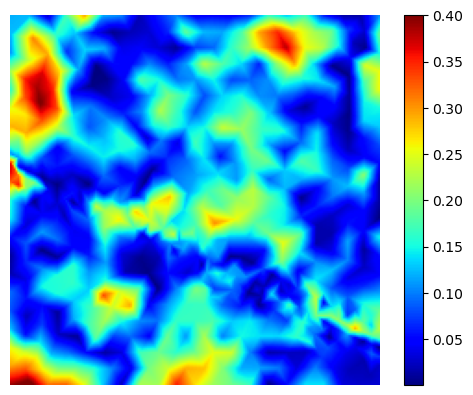}
    \caption*{bias}
\end{subfigure}
\caption{Posterior statistics of the doping profile $C$ for the AAO approach with $30,000$ posterior samples in three scenarios: baseline prior with MwG sampler, baseline prior with joint sampler and tightened prior with MwG sampler. \emph{Top row}: true $C$ and posterior means. \emph{Middle row}: posterior variances. \emph{Bottom row}: pixel--wise bias. MwG and joint samplers achieve similar posterior mean and variance. The tightened prior helps to reduce posterior variance.}
\label{fig: semiconductor; informative prior and joint sampler}
\end{figure}

\begin{table}[h]
    \centering
    \caption{Comparison of computational time (in minutes), bias, variance, MSE and acceptance rate for the AAO approach with MwG and joint sampler, both for the baseline prior.}
    \label{tab: joint sampler}
    \begin{tabular}{c c c c c c c c c}
        \hline
        & sampler & samples & time & bias & variance & MSE & \multicolumn{2}{c}{acceptance rate} \\
        \cline{8-9}
        & & & & & & & $\theta$ & $u$ \\
        \hline
        AAO & MwG   & 30,000 & 18 & 0.0197 & 0.2134 & 0.2332 & 70\% & 12.3\% \\
        AAO & joint & 30,000 & 17 & 0.0213 & 0.2075 & 0.2280 & \multicolumn{2}{c}{11.0\%} \\
        \hline
    \end{tabular}
\end{table}

%%%%%%%%%%%%%%%%%%%%%%%%%%%%%%%%%%%%%%%%%%%%%%%%%%%%%%%%%%%%%%%%%
\section{Conclusions}\label{s:conclusions}
In this work, we proposed a new all-at-once Bayesian inversion approach to reduce the computational costs 
when dealing with complex PDE models.  
Inspired by all-at-once optimization, we introduce a hybrid likelihood function comprising a data-misfit term and a PDE residual term. 
This leads to the treatment of the PDE residual as a probabilistic object %into the likelihood 
with a residual covariance, enabling sampling the unknown parameter and state simultaneously. 
Based on this, a Metropolis-within-Gibbs MCMC  and a joint sampling algorithms with PDE-informed prior
were proposed to speed up the posterior sampling process.
  
Numerical experiments show that the new Bayesian all-at-once approach outperforms the the conventional reduced setting in terms of computational time, while maintaining comparable inference accuracy.

Future work includes the extension of the all-at-once framework to bi-level settings \cite{nguyen24,nguyen25} for Bayesian inversion. {This extension leads to investigate how the hierarchical structure of bi-level setting can be incorporated into the Bayesian formulation, particularly in terms of constructing informative priors and designing efficient posterior sampling methods. To this end, developing hierarchical or conditional probabilistic models that account for the uncertainty introduced at both levels will enable more comprehensive quantification of uncertainties for the inferred parameters and PDE states. Furthermore, bi-level formulation may also motivate new blockwise Metropolis-within-Gibbs schemes that leverage the conditional structure of the posterior and reduce the cost of sampling. These developments will lead to a comprehensive statistical framework for scalable Bayesian inversion in complex PDE-constrained inverse problems.}

\section*{Acknowledgment} TN acknowledges the support of Cambridge Hub for Innovative Mathematics in Research and Application - CHIMiRA, Gates Foundation [grant number AW00262].
LT acknowledges support from the Austrian Science Fund (FWF), Elise-Richter grant DOI 10.55776/V1000. FS acknowledges funding support from the European Research Council (ERC) under the Horizon 2020 research and innovation program of the European Union (Grant agreement No. 101125225) as well as from the Austrian Science Fund (FWF) grant I6667-N.

\bibliographystyle{abbrv}
\bibliography{references}

\end{document}